\documentclass[12pt,letterpaper,twoside,english]{amsart}
\usepackage[T1]{fontenc}
\usepackage[latin9]{inputenc}
\usepackage{babel}
\usepackage{mathrsfs}
\usepackage[headings]{fullpage}
\usepackage{tikz}
\usetikzlibrary{decorations.pathmorphing}
\usetikzlibrary{arrows}

\def\centerarc[#1](#2)(#3:#4:#5);%
    {
    \draw[#1]([shift=(#3:#5)]#2) arc (#3:#4:#5);
    }

\usepackage{amsthm}
\usepackage{amstext}
\usepackage{amssymb}
\usepackage{graphicx}
\usepackage{setspace}
\usepackage{esint}
\usepackage{comment}
\usepackage[all,cmtip]{xy}
\usepackage[bbgreekl]{mathbbol}
\usepackage{amsfonts}
\DeclareSymbolFontAlphabet{\mathbb}{AMSb}
\DeclareSymbolFontAlphabet{\mathbbl}{bbold}
\usepackage[shortlabels]{enumitem}
\usepackage{amscd,color}
\usepackage[unicode=true,pdfusetitle,bookmarks=true,bookmarksnumbered=false,bookmarksopen=false,
 breaklinks=false,pdfborder={0 0 1},backref=false,colorlinks=false]
 {hyperref}

\makeatletter

\pdfpageheight\paperheight
\pdfpagewidth\paperwidth

\numberwithin{equation}{section}
\numberwithin{figure}{section}
\theoremstyle{plain}
\newtheorem{thm}{\protect\theoremname}[section]
  \theoremstyle{definition}
  \newtheorem{example}[thm]{\protect\examplename}
  \theoremstyle{remark}
  \newtheorem{rem}[thm]{\protect\remarkname}
  \theoremstyle{plain}
  \newtheorem{prop}[thm]{\protect\propositionname}
  \theoremstyle{plain}
  \newtheorem{cor}[thm]{\protect\corollaryname}
  \theoremstyle{definition}
  \newtheorem{defn}[thm]{\protect\definitionname}
  \theoremstyle{plain}
  \newtheorem{lem}[thm]{\protect\lemmaname}
  \theoremstyle{plain}
  \newtheorem{conj}[thm]{\protect\conjecturename}
\theoremstyle{plain}
  \newtheorem{prob}[thm]{\protect\problemname}
\theoremstyle{plain}
\newtheorem*{thm*}{\protect\theoremname}
\theoremstyle{definition}

\theoremstyle{definition}
\newtheorem*{defn*}{\protect\definitionname}
\theoremstyle{definition}
\newtheorem*{thx}{Acknowledgments}

  \newtheoremstyle{dotless}{}{}{\itshape}{}{\bfseries}{}{ }{}
  \theoremstyle{dotless}
  \newtheorem{conjec}[thm]{\protect\conname}

\usepackage{enumitem}
\usepackage{amsmath}
\usepackage{amsfonts}
\usepackage{amssymb}
\usepackage[all]{xy}
\usepackage{array}

\def\RR{\mathbb{R}}
\def\CC{\mathbb{C}}
\def\QQ{\mathbb{Q}}
\def\PP{\mathbb{P}}
\def\ZZ{\mathbb{Z}}
\def\HH{\mathbb{H}}
\def\DD{\mathbb{D}}
\def\VV{\mathbb{V}}

\def\aff{\mathbb{A}}
\def\ay{\mathbf{i}}

\def\cE{\mathcal{E}}
\def\cx{\mathcal{X}}
\def\cs{\mathcal{S}}
\def\ch{\mathcal{H}}

\def\ci{\mathcal{I}}

\def\co{\mathcal{O}}
\def\cv{\mathcal{V}}
\def\cu{\mathcal{U}}
\def\cz{\mathcal{Z}}
\def\cC{\mathcal{C}}

\def\gr{\mathrm{Gr}}

\def\IH{\mathrm{IH}}

\def\I{\mathrm{I}}

\def\um{\underline{m}}

\def\uo{\underline{0}}
\def\uz{\underline{z}}
\def\ux{\underline{x}}
\def\ua{\underline{a}}
\def\uL{\underline{\lambda}}
\def\ul{\underline{\ell}}
\def\ut{\underline{t}}

\def\tx{\tilde{X}}
\def\xc{\mathrm{X}^{\circ}}
\def\NP{\mathbbl{\Delta}}

\makeatother

  \providecommand{\corollaryname}{Corollary}
  \providecommand{\definitionname}{Definition}
  \providecommand{\examplename}{Example}
  \providecommand{\lemmaname}{Lemma}
  \providecommand{\propositionname}{Proposition}
 \providecommand{\problemname}{Problem}
  \providecommand{\remarkname}{Remark}
\providecommand{\theoremname}{Theorem}
\providecommand{\conjecturename}{Conjecture}
\providecommand{\conname}{Conjecture}

\newcommand{\comb}[2]{\left( {#1} \atop {#2} \right)}

\begin{document}

\title{Ap\'ery Extensions}

\author[V. Golyshev]{Vasily Golyshev}
\address{ICTP Math Section, Strada Costiera 11, Trieste 34151, Italy; and Algebra and Number Theory Lab, Institute for Information Transmission Problems, Bolshoi Karetny 19, Moscow 127994, Russia}
\email{golyshev@mccme.ru}

\author[M. Kerr]{Matt Kerr}
\address{Department of Mathematics, Washington University in St.~Louis, 1 Brookings Drive, Campus Box 1146, St. Louis, MO 63130-4899}
\email{matkerr@wustl.edu}

\author[T. Sasaki]{Tokio Sasaki}
\address{Department of Mathematics, University of Miami, 1365 Memorial Drive, Ungar building, Coral Gables, FL 33146-2508}
\email{txs885@miami.edu}

\begin{abstract}
The \emph{Ap\'ery numbers} of Fano varieties are asymptotic invariants of their quantum differential equations.  In this paper, we initiate a program to exhibit these invariants as (mirror to) limiting extension classes of higher cycles on the associated Landau-Ginzburg models --- and thus, in particular, as periods.  We also construct an \emph{Ap\'ery motive}, whose mixed Hodge structure is shown, as an application of the decomposition theorem, to contain the limiting extension classes in question.

Using a new technical result on the inhomogeneous Picard-Fuchs equations satisfied by higher normal functions, we illustrate this proposal with detailed calculations for LG-models mirror to several Fano threefolds.  By describing the ``elementary'' Ap\'ery numbers in terms of regulators of higher cycles (i.e., algebraic $K$-theory/motivic cohomology classes), we obtain a satisfying explanation of their arithmetic properties.  Indeed, in each case, the LG-models are modular families of $K3$ surfaces, and the distinction between multiples of $\zeta(2)$ and $\zeta(3)$ (or $(2\pi\mathbf{i})^3$) translates ultimately into one between algebraic $K_1$ and $K_3$ of the family.
\end{abstract}

\maketitle
\section{Introduction}

Over the years since Ap\'ery's discovery of a proof of irrationality of $\zeta(3)$
in 1979, his recursions 
$$n^3u_n-(34n^3-51n^2+27n-5)u_{n-1}+(n-1)^3 u_{n-2}=0, \eqno (A3)$$
and
$$n^2u_n-(11n^2-11n+3) u_{n-1}-(n-1)^2 u_{n-2}=0, \eqno (A2)$$
originally viewed as interesting individuals 
rather than members of an important species --- 
to follow Siegel's distinction ---
have gradually come to be seen 
as being the two most basic instances of Fuchsian DEs
arising from  
the class of extremal Calabi--Yau pencils which appears  in
various subjects in algebraic and arithmetic geometry. This makes them 
natural models for case studies of this class.
In each of the two  cases, 
a unique normalized integral solution $\{ a_n \}$ exists, namely 
$a_n^{(A3)}= \sum_{k=0}^n \comb{n}{k}^2 \comb{n+k}{k}^2 $, and
$a_n^{(A2)}= \sum_{k=0}^n \comb{n}{k}^2 \comb{n+k}{k}$.  By considering 
the solutions $\{ b_n \}$ with 
the initial conditions
$b_0 =0, \, b_1 =1$, and the rate of convergence of the ratios $b_n/a_n$ 
to the limits 
$$\lim_{n \to \infty} \dfrac{b_n^{(A3)}}{a_n^{(A3)}} = \frac{1}{6}\zeta(3) \;\;\;\;\;\text{and}\;\;\;\;\;
\lim_{n \to \infty} \dfrac{b_n^{(A2)}}{a_n^{(A2)}} =\frac{1}{5}\zeta(2),$$
Ap\'ery was able to deduce the irrationality of $\zeta (3)$ and improve 
the known order of irrationality of $\zeta (2)$.  Now, to make a story, we put $a_n^{(B3)} = \comb{2n}{n} a_n^{(A2)}$.
Clearly, the $\{a_n^{(B3)}\}$
satisfy
$$n^3u_n-(22n^3 - 33n^2 + 17n - 3)u_{n-1}+16 (n-3/2)(n-1)(n-1/2) u_{n-2}=0, \eqno (B3)$$
and the limit of the ratio $b_n^{(B3)}/a_n^{(B3)}$ is $\frac{1}{10}\zeta (2)$.
Passing to the generating series $\Phi(t) =  \sum c_n t^n$ with $c_n = a_n \text{ or } b_n$ 
and translating recursions into differential equations, we have  
$$ (D-1) L^{(A3)}   \Phi^{(A3)}(t)=0 \;\;\;\text{resp.}\;\;\;(D-1) L^{(B3)}  \Phi^{(B3)}(t)=0$$
with
$$
L^{(A3)} =
{D}^{3}-t\left (2\,D  +1 \right  )\left (
17\,{D}^{2}+17\,D+5\right )+{t}^{2}
\left (D+1\right )^{3}
$$
and  
$$
L^{(B3)}
= 
{D}^{3}-2\,t\left (2\,D + 1 \right )
\left (11\,{D}^{2}+11\,D+3\right )-4\,
{t}^{2}\left (D+1\right )\left (2\,D+3
\right )\left (2\,D + 1 \right )
$$

The differental operators $L^{(A3)}$ and $L^{(B3)}$ have quite a lot in common:

\begin{itemize}

\item  quite naively, both are linear third-order differential operators 
of degree 2 in $t$;

\item more specifically, both are operators of D3 type \cite{GolyshevStienstra2007}, 
which implies that the $i^{\text{th}}$ coefficient in $t$ is odd as a polynomial
in $D+i/2$;

\item  each is a Picard--Fuchs operator which controls the variation of polarized Hodge 
structure arising in the relative $H^2$ in a Shioda--Inose pencil of K3 surfaces;

\item both are \emph{extremal}, in the sense that the parabolic cohomology groups $\mathrm{IH}^1(\mathbb{P}^1,\text{--})$ of the associated local systems vanish;

\item both are modular, in the sense that in each case the solution 
$\sum_{n=0}^\infty  a_n t^n$ can be identified 
with the expansion  of a certain weight 2 level $N$  Eisenstein series with respect 
to a Hauptmodul on $X_0^*$ ($N=6 \text{ and } 5$, respectively);

\item each is the \emph{regularized quantum differential operator} of a Fano complete intersection in a $G/P$ (namely, a 7--fold hyperplane 
section of $OG(5,10)$ for $(A3)$ and an intersection of two hyperplanes and 
a quadric in $G(2,5)$ for $(B3)$).  
It is this interpretation of $(A3)$ and $(B3)$ as Picard--Fuchs operators 
in the Landau--Ginzburg models of  such `Mukai threefolds' that we 
are concerned with in the present paper.

\end{itemize}  
Yet, in spite of all these similarities, the \emph{Ap\'ery limit} $b_n/a_n$ is proportional 
to  $\zeta (3)$ in the first case resp. to $\zeta (2)$ in the second.   Two natural questions arise:
\smallskip

\bf (1) \rm  On the A--side of mirror symmetry, does the Ap\'ery limit reflect the geometry of a Fano?   
Inspecting the five cases  of Mukai threefolds $V_{2n}, \, n=5, \ldots, 9$, one notices that the cases 
with $\lim_{n\to\infty} b_n/a_n$ of weight 3 match the rational Fanos $V_{12}, \, V_{16}$ and $V_{18}$,
while those of weight 2 correspond to the irrational $V_{10}$ and $V_{14}$. 

\bf (2) \rm  On the B--side, what properties of the Landau--Ginzburg pencil control 
the weight of the Ap\'ery limit?

\smallskip
A conceptual answer to both questions, unconditional  for \bf (2) \rm 
but  HMS--contingent for \bf (1)\rm, is that it is ultimately the structure of the `south pole'
($t=\infty$)
fiber of the LG pencil that is responsible  for the weight of the Ap\'ery limit. Indeed, by HMS for Fanos, 
the structure of a Landau--Ginzburg model in the neighborhood of the south pole is expected 
to be related to the \emph{residual category} \cite{KS20}, i.e. the semiorthogonal complement to the subcategory 
of $D^b_{\mathrm{coh}} (F)$ generated by an incomplete exceptional collection. For a rational
threefold the properties of the residual category are closer, in a sense, to those of
the $D^b_{\mathrm{coh}}$ of a \emph{bona fide} algebraic variety than in the irrational cases.
One notices a similararity here with how the intermediate Jacobian of a threefold can 
have properties inconsistent with being the full motive of a curve.
Nevertheless, one can unconditionally relate the structure of  
 the south-pole fiber to the Lefschetz decomposition of the cohomology 
of the \emph{Fano unsections} -- ambient varieties in which our Fano
sits as a hyperplane section or a complete intersection. If the first non--trivial 
primitive class occurs in $H^4$, rather than in $H^6$ or higher, 
the Ap\'ery constant is going to have 
weight 2, in agreement with the gamma conjecture. 
    
The present paper offers mainly a Hodge--theoretic answer to \bf (2)\rm. 
In a nutshell, the structure of the south-pole fiber, or more precisely, 
the type of the south--pole LMHS, dictates the Hodge type of the 
Hodge module whose underlying differential operator is $(D-1)L$:
the $\zeta (3)$ (and $L(\chi_{-3},3)$) cases correspond to extensions 
encoded by a class in relative $H^3_{\mathcal{M}}(\mathrm{K3},\QQ (3))$,  
while the $\zeta (2)$ cases correspond to extensions
associated with $H^3_{\mathcal{M}}(\mathrm{K3},\QQ (2))$. In $K$--theoretic 
terms, the distinction is between the $K_3$ and $K_1$ groups 
of relative K3 surfaces respectively.  A more general, heuristic view
that suggests itself is this:  just as the Fano $F$
knows the quantum cohomology of its unsection, the LG-model of $F$
remembers the asymptotics of periods of the LG-model of its unsection, expressed
as solutions to \emph{inhomogeneous} PF-equations $L(-)=g(t)$ which underlie generalized
normal functions.

\medskip

\bf Arithmetic Mirror Symmetry. \rm   
Despite a smattering of examples in recent years \cite{MW,DK1,JW,DK2}, the role of algebraic cycles and their invariants in mirror symmetry remains something of a mystery.  The above discussion suggests a new link in the context of Fano/LG-model duality, whose formulation (backed up by nontrivial evidence) is a principal goal of this paper.

One of the features of \emph{local} mirror symmetry uncovered in \cite{DK2,BKV2} was the entrance of \emph{mixed} Hodge structures, whose extension classes are described on the B-model side by regulators on algebraic $K$-theory.  These same regulator classes, called \emph{higher normal functions} when they occur in families, are at the heart of the second author's interpretation \cite{Ke17} of Ap\'ery's irrationality proofs for $\zeta(2)$ and $\zeta(3)$.  It was in an effort to ``recombine'' this with the first author's enumerative, A-model interpretation \cite{Gv} of Ap\'ery's recurrence (see also \cite{Gk}), that the animating slogan of this paper suggested itself:

\vspace{3mm}
\noindent\textbf{Arithmetic Mirror Symmetry Problem:} \emph{For each Fano
$n$-fold $F$ admitting a toric degeneration, show that its Ap\'ery
numbers arise as 
\begin{itemize}
\item limits of \textup{(}classical and higher\textup{)} normal functions produced by cycles
on a 1-parameter family of CY $(n{-}1)$-folds defined over $\bar{\QQ}$,
\end{itemize}
together with 
\begin{itemize}
\item extension classes in the monodromy-invariant part of a limiting MHS of the family.
\end{itemize}
\vspace{3mm}}

\noindent 
A more detailed statement of this problem may be found in \S\ref{SIV2}.

While computations by G. da Silva \cite{dS2} appeared to support this line of thinking for the \emph{rational} Fano 3-folds in \cite{Gv}, there initially seemed to be little hope for the Ap\'ery numbers $\tfrac{1}{10}\zeta(2),\tfrac{1}{7}\zeta(2)$
of the non-rational Fanos $V_{10},V_{14}$, with neither the ``deresonation off the motivic setting'' nor the ``quantum Satake'' argument on the irregular side
offering an explanation of these numbers as periods. 
Moreover, the model of \cite{Ke17}, in its limitation to $K_n^{\text{alg}}$ of CY $(n-1)$-folds, could only produce rational multiples of $(2\pi \ay)^{3}$ or $\zeta(3)$ if $n=3$.  However, a new paradigm began to emerge around two years ago, allowing a much greater variety of cycles to enter.  Our main result is thus the following affirmative solution to the Arithmetic Mirror Symmetry Problem (more precisely, to Problem \ref{c52}):

\begin{thm}\label{tm}
The Ap\'ery numbers of the five Mukai Fano threefolds\footnote{These are, by definition, the rank-one Fano 3-folds arising as complete intersections in the Grassmannians of simple Lie groups other than projective spaces \cite{Gv}; they are $V_{10},\,V_{12},\,V_{14},\,V_{16}$, and $V_{18}$.} are limits of higher normal functions arising from motivic cohomology classes on associated Landau-Ginzburg models.
\end{thm}

The Theorem is proved in \S\S\ref{SIVA}-\ref{SIVC} (modulo a detail deferred to \S\ref{SV}), using a new result on inhomogeneous Picard-Fuchs equations satisfied by higher normal functions (Theorem \ref{t51}).  In \S6 we propose a theory of ``Ap\'ery extensions'' on the B-model side, which encompasses these examples, and highlight some implications of an affirmative solution to Problem \ref{c52}.  In \S\S\ref{SI}-\ref{SIII} we place our story in context, recalling the mixed Hodge theory of GKZ systems and local mirror symmetry, quantum $\mathcal{D}$-modules and Ap\'ery constants of Fanos, and higher normal functions on Landau-Ginzburg models.  In the rest of this Introduction, we would like to convey the idea of what an Ap\'ery extension is and why it is important.

Let $\cx\overset{\pi}{\to}\PP^{1}$ be a family
of compact CY $(n-1)$-folds with smooth total space and fibers\footnote{written $\tx_t$ in the body of the paper} $X_t=\pi^{-1}(t)$,
smooth off $\Sigma=\{0,t_{1},\ldots,t_{c},\infty\}$.
Laurent polynomials $\phi(\underline{x})\in\bar{\QQ}[x_{1}^{\pm1},\ldots,x_{n}^{\pm1}]$
with reflexive Newton polytope $\Delta$ are a key source for such
families, with $\cx$ obtained by blowing up $\PP_{\Delta}$
along $\overline{\{\phi=0\}}\cap(\PP_{\Delta}\backslash\mathbb{G}_{m}^{n})$,
and $\pi$ extending $1/\phi$. In particular, the \emph{mirror
LG-model of a Fano $F$ degenerating to} $\PP_{\Delta^{\circ}}$
(like those in \cite{Gv,Gk}) arises in this way.  

The cohomologies of the fibration $\mathcal{X}_{\cu}\overset{\pi_{\cu}}{\to}\cu:=\PP^{1}\backslash\Sigma$
produce VHSs $\ch^{\ell}=\ch_{f}^{\ell}\oplus\ch_{v}^{\ell}$
with ``fixed'' and ``variable'' parts. At each $\sigma\in\Sigma$,
we have the LMHS functor $\psi_{\sigma}$ and monodromies $T_{\sigma}$.
All our families will have maximal unipotent monodromy at the ``north
pole'' $\sigma=0$;\footnote{``North'' refers to the infinity point of the Landau-Ginzburg potential; since we work primarily in a neighborhood of this point, however, it is $t=0$ for us.} for simplicity, here we also assume\footnote{in the Introduction, but not in the body of the paper} $\mathrm{rk}(T_{\sigma}-I)=1$
if $\sigma\neq0,\infty$, and $\ker\{H^{n}(X_{\sigma})\to\psi_{\sigma}\ch^{n}\}=\{0\}$
$(\forall\sigma)$.  Then we may write $\mathsf{A}^{\dagger}_{\phi}:=H^{n}(\cx\backslash X_0,X_{\infty})$
as an extension
\begin{equation} \label{eqapery}
0\to(\psi_{\infty}\ch_{v}^{n-1})^{T_{\infty}}\to\mathsf{A}^{\dagger}_{\phi}\to\mathrm{IH}^{1}(\PP^{1}\backslash\{0\},\ch_{v}^{n-1})\to0.
\end{equation}
of MHS.  Now the Ap\'ery numbers of $F$ record limits of ratios 
of solutions to its quantum difference equation (Definition \ref{d33a}); and a first approximation to the 
Problem is to find them in the extension class of \eqref{eqapery}.

Unfortunately, extension classes of MHS do not produce well-defined
numbers.  For instance, we have $\mathrm{Ext}_{\mathrm{MHS}}^{1}(\QQ(-a),\QQ(0))\cong\CC/\QQ(a)$,
which (say) would make $\tfrac{1}{10}\zeta(2)$ trivial in $\CC/\QQ(2)$.
This is where writing them \emph{as limits of admissible normal functions}\footnote{Admissible normal functions are reviewed in \S\ref{SIIIB} below.} enters:
if \eqref{eqapery} arises as $\lim_{t\to 0}\nu(t)$
for some $\nu\in\mathrm{ANF}(\ch_{v}^{n-1}(r))$, and $k:=\mathrm{rk}((\psi_{0}\ch_{v}^{n-1})^{T_{0}})$,
then $\nu$ has a \emph{unique} lift $\tilde{\nu}$ on the disk $|t|<|t_{k+1}|$
to a single-valued section of $\mathcal{H}_{v}^{n-1}$. Pairing this
with a suitable section $\omega\in\Gamma(\PP^{1},\mathcal{F}^{n-1}\mathcal{H}_{v,e}^{n-1})$
yields a \emph{truncated higher normal function} (THNF) $V(t)=\langle\tilde{\nu},\omega\rangle$ whose
first $k$ Taylor coefficients in $t$ are well-defined
complex numbers refining the information in $\mathsf{A}^{\dagger}_{\phi}$.
So \emph{higher normal functions get us from extension data to constants}, and this
is why it is better to state the Problem in terms of their limits.  But which HNF to choose?  Here are two candidates.

Consider the VMHS $\mathcal{A}_{\phi}^{\sigma}:=H^{n}(\cx\backslash X_{\sigma},X_t)$
over $U$ ($\sigma=0$ or $\infty$). As an extension it reads\begin{equation}\label{ext} 0\to\ch_{v}^{n-1}\to\mathcal{A}_{\phi}^{\sigma}\to\mathrm{IH}^{1}(\PP^{1}\backslash\{\sigma\},\ch_{v}^{n-1})\to0 ,
\end{equation}in which the $\mathrm{IH}$ term is a \emph{constant} VMHS. Taking
first $\sigma=0$, $\mathsf{A}^{\dagger}_{\phi}=(\psi_{\infty}\mathcal{A}_{\phi}^0)^{T_{\infty}}$
is recovered as the ``south pole'' limit of $\mathcal{A}_{\phi}^{0}$.
If $\ch_{v}^{n-1}$ is extremal (cf. \S\ref{SV}) with Hodge numbers all $1$, then $\mathrm{IH}^{1}(\PP^{1}\backslash\{0\},\ch_{v}^{n-1})\cong\QQ(-n)$,
and \eqref{ext} gives a normal function in $\mathrm{ANF}(\ch_{v}^{n-1}(n))$.
According to the Beilinson-Hodge conjecture (cf. Conjecture \ref{c42a} below), this should come from a ``$K_{n}$'' cycle (in $\mathrm{CH}^{n}(\cx\setminus X_{0},n)$),
recovering the paradigm of \cite{Ke17}.

The alternate ($\sigma=\infty$) perspective is to view $\mathsf{A}^{\dagger}_{\phi}=[(\psi_{0}\mathcal{A}_{\phi}^{\infty})^{T_{0}}]^{\vee}(-n)$
as a ``north pole'' limit. 
This can make a huge difference, since
$\mathcal{A}_{\phi}^{\infty}$ and $\mathcal{A}_{\phi}^{0}$
\emph{are not dual} in general (although the invariant parts of
their limits are). Indeed, for any morphism $\QQ(-a)\overset{\mu}{\hookrightarrow}\mathrm{IH}^{1}(\PP^{1}\backslash\{\infty\},\ch_{v}^{n-1})$,
the $\mu^{*}$-pullback 
\begin{equation}\label{apext}
0\to \ch^{n-1}_v\to \mu^*\mathcal{A}^{\infty}_{\phi}\to \QQ(-a)\to 0
\end{equation}
of \eqref{ext} belongs to $\mathrm{ANF}(\ch_{v}^{n-1}(a))$,
and (by Beilinson-Hodge again) should arise from a ``$K_{2a-n}$''
cycle (in $\mathrm{CH}^{a}(\cx\backslash X_{\infty},2a-n)$).  It is the
extensions of VMHS \eqref{apext} that we call \emph{Ap\'ery extensions} (Definition \ref{d6c}).
In the families of $K3$s mirror to $V_{10}$ and $V_{14}$, we have
$a=2$, and the corresponding normal functions do indeed arise from torically natural
$K_{1}$-cycles whose THNFs have the north pole limits $\tfrac{1}{10}\zeta(2),\tfrac{1}{7}\zeta(2)$.
This change in perspective came as a revelation since, for these and similar cases,
the south-pole approach is not computationally viable (Remark \ref{r6a}).


\begin{thx}
VG and MK are grateful to the members of the International groupe de travail on differential equations in Paris for the stimulating discussions of some of the material of this paper, and the Max Planck Institute for Mathematics for support during our visits there.  MK also thanks A. Harder, B. Lian and T. Pantev for relevant discussions.  This work was partially supported by Simons Collaboration Grant 634268 and NSF Grant DMS-2101482 (MK).
\end{thx}

\section{Generic Laurent polynomials} \label{SI}

\subsection{GKZ system} \label{SIA}

Fix a vector $\ua\in \CC^{N+1}$ and a convex polytope $\NP\subset \RR^{n+1}$ containing the origin, with vertices in $\ZZ^{N+1}$.  The corresponding toric variety $\PP_{\NP}$ compactifies $\mathbb{G}_m^{N+1}$ (with coordinate $\ux$).
Let $\mathsf{M}\subseteq \ZZ^{N+1}$ denote the monoid generated by $\mathfrak{M}:=\NP\cap(\ZZ^{N+1}\setminus\{\uo\})$ and $\mathbb{L}\subset \ZZ^{|\mathfrak{M}|}$ the lattice of relations; we assume for simplicity that $\mathsf{M}^{\text{gp}}=\ZZ^{N+1}$ and $\mathsf{M}=\ZZ^{N+1}\cap\text{Cone}_{\uo}(\NP)$.
The coefficients $\uL$ of the generic Laurent polynomial  $f(\ux)=\sum_{\um\in\mathfrak{M}} \lambda_{\um}\ux^{\um}$ parametrize the affine parameter space on which we define the \emph{GKZ system} of partial differential operators:
\begin{equation}\label{e21a}
\left\{ 
\begin{array}{cc}
Z_i = \sum_{\um\in\mathfrak{M}}m_i\delta_{\lambda_{\um}}+a_i & (i=0,\ldots,N)
\\
\square_{\ul}=\prod_{\ell_{\um}>0}\partial_{\lambda_{\um}}^{\ell_{\um}}-\prod_{\ell_{\um}<0}\partial^{-\ell_{\um}}_{\lambda_{\um}} & (\ul\in\mathbb{L}\subset \ZZ^{|\mathfrak{M}|})
\end{array}
\right.	
\end{equation}
\begin{prop}\label{p21a}
For each relative cycle $\mathscr{C}$ on $(\PP_{\NP}\setminus\{f=0\},\DD_{\NP}\setminus\{f=0\})$, the function
\begin{equation} \label{e21b}
\mathscr{P}_{\mathscr{C}}(\uL)=\int_{\mathscr{C}}\ux^{\ua}e^{f(\ux)}\mathrm{dlog}(\ux)
\end{equation}
is a (local) solution of \eqref{e21a}.
\end{prop}
\begin{proof}[Check:]
Applying $Z_i$ to $\mathscr{P}_{\mathscr{C}}$ gives
\[
\int_{\mathscr{C}}\ux^{\ua}(\delta_{x_i}f+a_i)e^f \mathrm{dlog}(\ux)=\int_{\mathscr{C}}d[\ux^{\ua}e^f\mathrm{dlog}(\ux_{\hat{i}})]=0,
\]
while applying $\square_{\ul}$ yields
\[
\int_{\mathscr{C}}\ux^{\ua}(\ux^{\sum_{\um\colon\ell_{\um}>0}\ell_{\um}\um}-\ux^{\sum_{\um\colon\ell_{\um}<0}(-\ell_{\um})\um})e^f\mathrm{dlog}(\ux)=0.
\]The solutions are (analytically) local because the cycles $\mathscr{C}$, and hence their period integrals $\mathscr{P}$, have monodromy about divisors in $\aff^{|\mathfrak{M}|}$.
\end{proof}
Since the corresponding $\mathcal{D}=\CC[\uL,\partial_{\uL}]$-module
\begin{equation} \label{e21c}
\tau_{\mathrm{GKZ}}^{\ua,\NP}:=\mathcal{D}/\mathcal{D}\langle\{Z_i\},\{\square_{\ul}\}\rangle
\end{equation}
is holonomic \cite[Thm 3.9]{Ad}, the (local) solutions module
\begin{equation} \label{e21d}
\mathrm{Hom}_{\mathcal{D}}(\tau,\hat{\co}_{\uL^{0}})\simeq \mathrm{Hom}_{\CC}(\CC_{\uL^0}\otimes _{\CC[\uL]}\tau,\CC)	
\end{equation}
at a point $\uL^0\in\CC^{|\mathfrak{M}|}$ is finite-dimensional.  
We shall think of \eqref{e21c} and \eqref{e21d} as ``cohomology'' and ``homology'' respectively, motivated by the parametrization of solutions by relative cycles; this will be made more precise in $\S$\ref{SIC}.

\subsection{Periods and residues} \label{SIB}

Given $\um\in\ZZ^{N+1}$, write $\deg(\um)=:\kappa$ for the minimal $\kappa\in\ZZ_{\geq 0}$ such that $\kappa\NP\ni \um$.  The ring $R=\CC[\uL][\ux^{\mathfrak{M}}]$, its Jacobian ideal $J_f=(\{\partial_{x_i}f\}_{i=0}^N )$, and the Jacobian ring $R/J_f$ are thereby graded by degree.  Moreover, sending $p(\ux)\mapsto p(\ux)\ux^{\ua}e^{f(\ux)}\mathrm{dlog}(\ux)$ induces a grading on $\tau_{\text{GKZ}}^{\NP}$ and a graded isomorphism
\begin{equation}\label{e22a}
\mathrm{gr}(R/J_f) \overset{\cong}{\underset{\mathrm{gr}(\mathcal{D})}{\longrightarrow}}\mathrm{gr}(\tau_{\text{GKZ}}^{\NP}).
\end{equation}
Specializing $\uL$ to a very general point $\uL^0$ (and hence $R$ to $R^0=\CC[\ux^{\mathfrak{M}}]$), the graded pieces have dimensions
\begin{equation}\label{e22b}
\dim_{\CC}(R^0/J_f)_{(k)}=\sum_{j=0}^{N+1}(-1)^j \textstyle{\binom{N+1}{j}}\dim(R^0_{(k-j)}),
\end{equation}
(where $\dim (R^0_{(k-j)})$ counts the points of degree $k-j$ in $\mathsf{M}$) with sum over $k$
\begin{equation}\label{e22c}
\dim_{\CC}(R^0/J_f)=(N+1)!\,\text{vol}(\NP).\end{equation}
See \cite[(5.3) and Cor. 5.11]{Ad}.

\subsubsection*{Irregular case in mirror symmetry:}
A polytope $\Delta\subset\RR^{n}$ with integer vertices is \emph{reflexive} iff its polar polytope $\Delta^{\circ}$ also has integer vertices; this implies that both have $\uo$ as unique interior integer point.  Fixing such a $\Delta$, take $\NP:=\Delta$ and $\ua=\uo$.  (Note that $N=n-1$.)  Then we have a graded isomorphism of A- and B-model $\mathcal{D}$-modules
\begin{equation}\label{e22d}
QH^*(\PP_{\Delta^{\circ}})	\underset{\mathrm{gr}}{\cong}\tau^{\Delta}_{\text{GKZ}}
\end{equation}
with the grading by $\tfrac{\deg}{2}$ on the left-hand side (see \cite{Ir2}).

\begin{example}\label{ex22a}
Let $\Delta$ be the triangle in the figure.  Choose $\uL$ so that the cycle $\mathbb{T}^2 \cong S^1\times S^1$ 
\\
\noindent\begin{minipage}{0.19\textwidth}
\begin{tikzpicture}
\draw[black,thick,fill=cyan] (-1,-1) -- (1,0) -- (0,1) -- (-1,-1);
\draw[gray,thick] (-1.3,0) -- (1.3,0);
\draw[gray,thick] (0,-1.3) -- (0,1.3);
\filldraw[black] (-1,-1) circle (2pt) node[anchor=east] {$\lambda_3$};
\filldraw[black] (1,0) circle (2pt) node[anchor=north] {$\lambda_1$};
\filldraw[black] (0,1) circle (2pt) node[anchor=east] {$\lambda_2$};
\end{tikzpicture}
\end{minipage}
\begin{minipage}{0.02\textwidth}
\end{minipage}
\begin{minipage}{0.79\textwidth}
given by $|x_1|=|x_2|=1$ avoids the zero-locus of $f=\lambda_1 x_1+\lambda_2 x_2+\lambda_3 x_1^{-1}x_2^{-1},$ so that $\mathscr{C}=\mathbb{T}^2$ is a relative cycle.  By \eqref{e22b}-\eqref{e22c}, the rank of $\tau$ is 3, with three graded pieces each of rank 1.  Computing the period in \eqref{e21b} now gives
\begin{align*}
\tfrac{1}{(2\pi\ay)^2}\mathscr{P}=\tfrac{1}{(2\pi\ay)^2}\int_{\mathbb{T}^2} e^f \mathrm{dlog}(\ux) = \tfrac{1}{(2\pi\ay)^2}\sum_{n\geq 0}\tfrac{1}{n!}\int_{\mathbb{T}^2}f^n \mathrm{dlog}(\ux) =\sum_{m\geq 0}\tfrac{(\lambda_1\lambda_2\lambda_3)^m}{(m!)^3},
\end{align*}
\end{minipage}

\noindent which is an irregular/exponential period.  In particular, we see that $\tau_{\text{GKZ}}^{\Delta}$ does not underlie a classical VHS or VMHS.
\end{example}

\subsubsection*{Regular case in mirror symmetry:}
With $\Delta$ as above, take $\NP\subset \RR^{1+n}$ to be the convex hull of the origin and $\{1\}\times \Delta$ (so that now $N=n$), and put $\ua:=(1,\uo)$.  We denote the resulting GKZ system by $\hat{\tau}_{\text{GKZ}}^{\Delta}$.  It has the same rank as $\tau_{\text{GKZ}}^{\Delta}$ since $\text{vol}(\NP)=\tfrac{1}{(n+1)!}\text{vol}(\Delta)$.  Rather than being isomorphic, the two are related (roughly) by Fourier-Laplace transform; and (as will be explained in \S\ref{SIC}) we have an isomorphism of $\mathcal{D}$-modules
\begin{equation}\label{e22e}
QH_c^{*(+2)}(K_{\PP_{\Delta^{\circ}}})\underset{\text{gr}}{\cong}\hat{\tau}_{\text{GKZ}}^{\Delta}	,
\end{equation}
where $K_{\PP_{\Delta^{\circ}}}$ is the total space of the canonical line bundle on $\PP_{\Delta^{\circ}}$.

Now let $\phi_{\uL}(\ux)=\sum_{m\in \Delta\cap\ZZ^n} \lambda_{\um}\ux^{\um}$ be a general Laurent polynomial on $\Delta$ and
$\Gamma$ a relative $n$-cycle in $(\PP_{\Delta}\setminus\{\phi=0\},\DD_{\Delta}\setminus\{\phi=0\})$.  With $f=x_0\phi(\ux)$ and $\mathscr{C}=\RR_-\times \Gamma$, the periods in \eqref{e21b} take the form
\begin{equation}\label{e22f}
\mathscr{P}=\int_{\mathscr{C}}x_0 e^f \tfrac{dx_0}{x_0}\wedge\mathrm{dlog}(\ux)=\int_{\Gamma}\left(\int_{-\infty}^0 e^{x_0\phi}dx_0\right)\mathrm{dlog}(\ux) = \int_{\Gamma}\tfrac{\mathrm{dlog}(\ux)}{\phi(\ux)}
= 2\pi\ay\int_{\gamma}\mathrm{Res}_{\phi=0}\left(\tfrac{\mathrm{dlog}(\ux)}{\phi(\ux)}\right)\;
\end{equation}
$\text{if}\;\Gamma=\text{Tube}(\gamma)\;\text{for}\;\gamma\subset\{\phi=0\}.$  In particular, these are (regular) periods of a variation of mixed Hodge structure.

\begin{example}\label{ex22b}
With $\Delta$ as in Ex. \ref{ex22a}, $\NP$ is the tetrahedron in the figure. Taking $\Gamma=\mathbb{T}^2$ in
\\
\noindent\begin{minipage}{0.28\textwidth}
\begin{tikzpicture}
\draw[black,thick,fill=cyan] (-1.5,-0.77) -- (0,-1) -- (1,0);
\draw[black,thick,fill=cyan] (-1.5,-0.77) -- (1,0) -- (0.5,0.77) -- (-1.5,-0.77);
\draw[gray,thick] (-1.3,-1) -- (1.3,-1);
\draw[gray,thick] (-0.5,-1.77) -- (0.5,-0.23);
\draw[gray,thick] (0,-1.8) -- (0,0.7);
\filldraw[black] (-1.5,-0.77) circle (2pt) node[anchor=east] {$\lambda_3$};
\filldraw[black] (1,0) circle (2pt) node[anchor=west] {$\lambda_1$};
\filldraw[black] (0.5,0.77) circle (2pt) node[anchor=west] {$\lambda_2$};
\filldraw[black] (0,0) circle (2pt) node[anchor=east] {$\lambda_0$};
\draw[black,thick,dotted] (0.5,0.77) -- (0,-1);
\end{tikzpicture}
\end{minipage}
\begin{minipage}{0.72\textwidth}
\eqref{e22f} and writing $t=\tfrac{\lambda_1\lambda_2 \lambda_3}{\lambda_0^3}$, Cauchy residue gives for $|t|<\tfrac{1}{27}$
\begin{align*}
\tfrac{1}{(2\pi\ay)^2}\mathscr{P}=\int_{\mathbb{T}^2}\frac{\mathrm{dlog}(\ux)/(2\pi\ay)^2}{\lambda_0 + (\lambda_1 x_1 + \lambda_2 x_2 + \lambda_3 x_1^{-1}x_2^{-1})}= \frac{1}{\lambda_0}\sum_{m\geq 0}\frac{(3m)!}{(m!)^3} t^m.
\end{align*}
Since $\mathrm{rk}(\hat{\tau}^{\Delta}_{\text{GKZ}})=3$, one expects 3 distinct periods related to the geometry of the family of elliptic curves $E_t=\overline{\{\phi_{\uL}(\ux)=0\}},$ which
\end{minipage}

\noindent has a type $\mathrm{I}_9$ singular fiber at $t=0$.  

Fix $\lambda_0=1$.  If $\{\alpha,\beta\}$ is a symplectic basis for $H_1(E_t)$, with $\alpha$ vanishing at $t=0$, we can take $\Gamma$ to be $\text{Tube}(\alpha)\simeq \mathbb{T}^2$ ($\mathscr{P}$ holomorphic in $t$), $\text{Tube}(\beta)$ ($\mathscr{P}\sim \tfrac{9}{2\pi\ay}\log(t)$), or $\sigma:=\RR_- \times \RR_-$ ($\mathscr{P}\sim \tfrac{9}{2(2\pi\ay)^2}\log^2(t)$).  Note that only the first two are ``periods of $E_t$''.
\end{example}

\subsection{Mixed Hodge theory of GKZ} \label{SIC}

Continuing with the ``regular case'' above, and recalling that $\DD_{\Delta}:= \PP_{\Delta}\setminus \mathbb{G}_m^n$, we set $X_{\uL}:=\overline{\{\phi_{\uL}(\ux)=0\}}$ and $\partial X_{\uL}:=X_{\uL}\cap \DD_{\Delta}$.  By Prop. \ref{p21a}, we know that at least some solutions of $\hat{\tau}_{\text{GKZ}}^{\Delta}$ are parametrized by the choice of $\Gamma\in H_n(\PP_{\Delta}\setminus X_{\uL},\DD_{\Delta}\setminus\partial X_{\uL})$, which (as a best-case scenario) suggests the following
\begin{thm}[\cite{HLYZ}]\label{t23a}
We have a canonical isomorphism
\begin{equation}\label{e23a}
\hat{\tau}_{\text{GKZ}}^{\Delta}\cong H^n(\PP_{\Delta}\setminus X_{\uL},\DD_{\Delta}\setminus \partial X_{\uL}),
\end{equation}
in which the $\mathcal{D}$-module structure on the RHS is defined by the Gauss-Manin connection.
\end{thm}
The connection to mirror symmetry is amplified by
\begin{thm}[Conjectured by \cite{KKP}, proved by \cite{Ha} ($n=3$) and \cite{Sa}]\label{t23b}
$$
\dim \gr^{n-k}_F H^n(\PP_{\Delta}\setminus X_{\uL},\DD_{\Delta}\setminus\partial X_{\uL})=\dim H^{k,k}(\PP_{\Delta^{\circ}})=\dim H^{k+1,k+1}_c(K_{\PP_{\Delta^{\circ}}}).
$$
\end{thm}
\noindent This refines \eqref{e23a} into a graded isomorphism strongly reminiscent of Griffiths's residue theory \cite{Gr69}, with $\mathrm{gr}_k$ (resp. multiplication by $x_0 \ux^{\um}$ as a map from $\mathrm{gr}_k \to \mathrm{gr}_{k+1}$) on the left matching $\gr_F^{n-k}$ (resp. $\nabla_{\partial_{\lambda_{\um}}}\colon \gr_F^{n-k}\to \gr_F^{n-k-1}$) on the right.  

However, the RHS of \eqref{e23a} is a (variation of) \emph{mixed} Hodge structure, with a nontrivial \emph{weight} filtration.  While intersection theory on the A-model $K_{\PP_{\Delta^{\circ}}}$ allows us to compute a basis of solutions to GKZ via mirror symmetry (Thm. \ref{t23c} below), it is unclear how to see the weight filtration directly in these terms.  To elaborate, we pose two questions:

\subsubsection*{(1) How might one isolate the highest weight part $\gr_{n+1}^W$ of \eqref{e23a} \textup{(}i.e., $H^{n-1}(X_{\uL})$\textup{)} within the setting of GKZ solutions?}~\\~\\
\noindent Under mirror symmetry we have the correspondences:
\begin{itemize}[leftmargin=0.5cm]
\item	$\um\in\mathfrak{M}$ $\longleftrightarrow$ divisors $[D_{\um}]\in H^2(\PP_{\Delta^{\circ}})$;
\item  relations $\ell\in \mathbb{L}$ $\longleftrightarrow$ curves $[C_{\ul}]\in H_2(\PP_{\Delta^{\circ}})$; and
\item  \emph{Mori cone} $\mathbb{L}_{\geq 0}\subset \mathbb{L}$ $\longleftrightarrow$ effective curve classes.
\end{itemize}
We assume $\mathbb{L}_{\geq 0}$ is simplicial with basis $\{\ul^{(i)}\}$, and put $t_i :=\uL^{\ul^{(i)}}$ (e.g. $t=\tfrac{\lambda_1\lambda_2\lambda_3}{\lambda_0^3}$ above), $\tau_i:=\tfrac{\log(t_i)}{2\pi\ay}$.  The isomorphism class of $X_{\uL}(=:X_{\ut})$ depends only on $\ut$.
\begin{thm}[\cite{HLY}]\label{t23c}
The \textup{(}$\CC$-linear combinations of\textup{)} periods $\mathscr{P}$ of $\hat{\tau}_{\text{GKZ}}^{\Delta}$ are the \textup{(}$\CC$-linear combinations of\textup{)} coefficients of cohomology classes in
$$
\mathscr{B}_{\Delta}:=\sum_{\ell\in\mathbb{L}_{\geq 0}}	\tfrac{\prod_{\um\colon\ell_{\um}<0}D_{\um}(D_{\um}-1)\cdots (D_{\um}+\ell_{\um}+1)}{\prod_{\um\colon\ell_{\um}>0}(D_{\um}+1)\cdots (D_{\um}+\ell_{\um})}(D_{\uo}-1)\cdots(D_{\uo}+\ell_{\uo})\uL^{\ul+\underline{D}}\in H^*(\PP_{\Delta^{\circ}})\otimes \CC[[\underline{t}]][\underline{\tau}].
$$
\end{thm}
\begin{conj}[Hyperplane Conjecture \cite{HLY,LZ}]
The periods of \textup{(}$\nabla$-flat sections of\textup{)} $H^{n-1}(X_{\ut})$ are the \textup{(}$\CC$-linear combinations of\textup{)} coefficients of cohomology classes in $\mathscr{B}_{\Delta}\cup [X^{\circ}]$, where $X^{\circ}\subset \PP_{\Delta^{\circ}}$ is an anticanonical hypersurface.
\end{conj}
\begin{example}\label{ex23a}
With $\Delta$ as in Examples \ref{ex22a}-\ref{ex22b}, we have $\PP_{\Delta^{\circ}}=\PP^2$, $[X^\circ]=3[H]$ (for $H$ a hyperplane in $\PP^2$), and $$\mathscr{B}_{\Delta}=[1](\text{holo. period})+[H](\log\text{ period})+[H]^2(\log^2\text{ period}).$$ In this case $$\mathscr{B}_{\Delta}\cup[X^{\circ}]=[H](\text{holo. period})+[H]^2(\text{log period}),$$ and so the hyperplane conjecture correctly asserts that the holomorphic and log periods are the actual periods of $H^1(E_t)$.
\end{example}

\subsubsection*{(2) Can we compute the remaining GKZ periods, especially those which yield extension classes of $\gr_{n+1}^W$ by other weight-graded pieces?}~\\~\\
\noindent Here ``compute'' means \emph{using the A-model}.  We know at present of no (even conjectural) \emph{intrinsic} A-model description of the full weight filtration.  An \emph{extrinsic} one, which we shall now sketch, was obtained in \cite{BKV2} by presenting $K_{\PP_{\Delta^{\circ}}}$ as the large-fiber-volume limit of compact elliptically-fibered Calabi-Yau $(n+1)$-folds.  (Though [op. cit.] treats the case $n=2$, this works in general.) To obtain these families of higher-dimensional CYs, let $\diamond\subset \RR^2$ be the convex hull of $\{(-1,1),(-1,-1),(2,-1)\}$, and $\hat{\Delta}\subset \RR^{n+2}$ be the convex hull of $\Delta\times {(-1,-1)}$ and $\uo\times \diamond$.  There are torically-induced morphisms $\PP_{\hat{\Delta}}\to \PP_{\Delta}$ and $\PP_{\hat{\Delta}^{\circ}}\to \PP_{\Delta^{\circ}}$ which restrict to elliptic fibrations on anticanonical (CY-)hypersurfaces $\hat{X},\hat{X}^{\circ}$.

In particular, write $\hat{X}_{\ut,s}$ for the closure of the zero-locus of $\Phi(\ux,u,v):=\mathsf{a}+\mathsf{b}u^2 v^{-1}+\mathsf{c}u^{-1}v^{-1}+\phi_{\uL}(\ux)u^{-1}v^{-1}$, where $s:=\tfrac{\lambda_0 \mathsf{b}^2 \mathsf{c}^3}{\mathsf{a}^6}$.  Instead of the large complex-structure limit ($\ut\to \uo$ and $s\to 0$), we take only $s\to 0$.  This has the effect of degenerating the generic fiber of $\hat{X}_{\ut,s}\to \PP_{\Delta}$ and decompactifying that of $\hat{X}^{\circ}\to \PP_{\Delta^{\circ}}$, resulting in the diagram
\begin{equation}\label{e23b}
\xymatrix@R-1pc@C-1pc{
& K_{\PP_{\Delta^{\circ}}} \ar @{->>}[ld]_{\aff^1} & & & & \hat{X}_{\ut,0} \ar @{->>} [rd]^{\mathrm{I}_5} \\
\PP_{\Delta^{\circ}} & & \ar @{<~>} [rr]^{\text{Batyrev}}_{\text{mirror}} & & & & \PP_{\Delta}\\
& \mspace{30mu}\hat{X}^{\circ}\subset \PP_{\hat{\Delta}^{\circ}} \ar @{->>} [lu]^{\text{ell.}} \ar @{-->} [uu]^{s\to 0} & & & & \PP_{\hat{\Delta}}\supset \hat{X}_{\ut,s}\mspace{30mu} \ar @{->>} [ru]_{\text{ell.}} \ar @{-->} [uu]_{s\to 0}
}\end{equation}
with solid arrows labeled by generic fiber type.  The singular CY $\hat{X}_{\ut,0}$ has the \emph{Hori-Vafa model} $$Y_{\ut}:=\hat{X}_{\ut,0}\cap (\mathbb{G}_m^n \times \aff^2) =\{\phi_{\uL}(\ux)+uv=0\},$$ a smooth noncompact CY $(n+1)$-fold, as a Zariski open subset, and one has the

\begin{thm}[\cite{DK1,BKV2}]\label{t23d}
There are isomorphisms of $\QQ$-VMHS
\begin{equation}\label{e23c}
H^{n+1}(\hat{X}_{\ut,0})\cong H_{n+1}(Y_{\ut})(-n-1)\cong H^n(\PP_{\Delta}\setminus X_{\ut},\DD_{\Delta}\setminus \partial X_{\ut}).
\end{equation}
\end{thm}
Now by Theorem \ref{t23a}, the RHS of \eqref{e23c} identifies with $\hat{\tau}_{\text{GKZ}}^{\Delta}$.  On the other hand, Iritani's results \cite{Ir} on $\hat{\Gamma}$-integral structure allow us to explictly compute the LHS of the isomorphism
\begin{equation}\label{e23d}
QH^{\text{even}}(\hat{X}^{\circ})\cong H^{n+1}(\hat{X}_{\ut,s})	
\end{equation}
of A- and B-model $\ZZ$-VHS.  Taking LMHS on both sides as $s\to 0$
\begin{equation}\label{e23e}
\psi_s QH^{\text{even}}(\hat{X}^{\circ})\cong \psi_s H^{n+1}(\hat{X}_{\ut,s}),
\end{equation}
the (unipotent) monodromy invariants
\begin{equation}\label{e23f}
QH^{*(+2)}_c(K_{\PP_{\Delta^{\circ}}})\cong H^{n+1}(\hat{X}_{\ut,0})	
\end{equation}
must agree as $\QQ$-VMHS.  
\begin{example}\label{ex23b}
For $\PP_{\Delta^{\circ}}=\PP^2$, the Hodge-Deligne diagrams\footnote{The number of dots in the $(p,q)$ spot represents $h^{p,q}$ of the given MHS.} for \eqref{e23d}-\eqref{e23f} are
\[\begin{tikzpicture}[scale=0.6]
\draw [thick,gray] (0,3.3) -- (0,0) -- (3.3,0);
\node [align=center] at (4.5,1.5) {{\tiny LMHS}\\ $\rightsquigarrow$\\ {\tiny $s\to 0$}};
\draw [thick,gray] (6,3.3) -- (6,0) -- (9.3,0);
\node [align=center] at (10.5,1.5) {\\ $\supset$\\ {\tiny $\ker(N)$}};
\draw [thick,gray] (12,3.3) -- (12,0) -- (15.3,0);
\filldraw[black] (0,3) circle (3pt);
\filldraw[black] (1.1,2) circle (3pt);
\filldraw[black] (2.1,1) circle (3pt);
\filldraw[black] (0.9,2) circle (3pt);
\filldraw[black] (1.9,1) circle (3pt);
\filldraw[black] (3,0) circle (3pt);
\filldraw[black] (6,0) circle (3pt);
\filldraw[black] (7,2) circle (3pt);
\filldraw[black] (8,2) circle (3pt);
\filldraw[black] (7,1) circle (3pt);
\filldraw[black] (8,1) circle (3pt);
\filldraw[black] (9,3) circle (3pt);
\node [blue] at (9,2.3) {$N$};
\draw [->,blue,thick] (7.8,1.8) -- (7.2,1.2); 
\draw [->,blue,thick] (6.8,0.8) -- (6.2,0.2);
\draw [->,blue,thick] (8.8,2.8) -- (8.2,2.2);
\filldraw[black] (12,0) circle (3pt);
\filldraw[black] (14,1) circle (3pt);
\filldraw[black] (13,2) circle (3pt);
\end{tikzpicture}\]
Here $n=2$ and $h^{2,1}(\hat{X}_{t,s}^{\circ})=2$, while each of the $\gr_F^k$ ($k=0,1,2$) in Theorem \ref{t23b} (visible in the right-most diagram) has rank 1.
\end{example}
The upshot is that we recover the isomorphism $QH^{*(+2)}_c(K_{\PP_{\Delta^{\circ}}})\cong \hat{\tau}_{\text{GKZ}}^{\Delta}$ claimed in \eqref{e22e}, while promoting it to an isomorphism of $\QQ$-VMHS.
Moreover, we obtain the promised A-model description of the weight filtration on $\hat{\tau}_{\text{GKZ}}^{\Delta}$ as the monodromy weight filtration $M_{\bullet}=W(N)[-n-1]_{\bullet}$ on LHS\eqref{e23f} $\subset$ LHS\eqref{e23e}.  We may therefore use \cite{Ir} to compute $W_{\bullet}\hat{\tau}_{\text{GKZ}}^{\Delta}$, and the associated ``mixed-$\QQ$-periods'', in terms of the intersection theory of $\PP_{\hat{\Delta}^{\circ}}$ and Gromov-Witten theory of $\hat{X}^{\circ}$, restricted to classes of curves whose volume remains finite in the $s\to 0$ limit.  This boils down to intersection theory and \emph{local} GW theory of $\PP_{\Delta^{\circ}}$.  (The reader who wants to see this worked out in detail in some $n=2$ cases may consult \cite{BKV2}.)

So far, we have said nothing about the extensions of MHS in \eqref{e23a} which these mixed periods are supposed to help us compute.  (For instance, the right-hand term of Example \ref{ex23b} can be viewed as the dual of the extension associated to the regulator of a family of $K_2^{\text{alg}}$ classes on the family $E_t$ of elliptic curves.) The analysis of these VMHS undertaken in $\S\S$\ref{SII}-\ref{SIII} works in a ``more general'' setting which allows us to \emph{drop} the genericity assumption on $\phi$.

\section{Special Laurent polynomials}\label{SII}
\subsection{Landau-Ginzburg models}\label{SIIA}

Instead of starting with a reflexive polytope and letting $\phi$ vary over the corresponding parameter space minus discriminant locus, we begin by fixing a Laurent polynomial $\phi\in \CC[x_1^{\pm 1},\ldots,x_n^{\pm 1}]$.  We assume that its Newton polytope $\Delta$ (the convex hull of those $\{\um\}$ for which $\ux^{\um}$ has nonzero coefficient) is reflexive; and fixing a maximal projective triangulation $\text{tr}(\Delta^{\circ})$, we also assume that the associated toric $n$-fold $\PP_{\Delta}:=\PP_{\Sigma(\text{tr}(\Delta^{\circ}))}$ is smooth.\footnote{We will write $\PP'_{\Delta}:=\PP_{\Sigma(\Delta^{\circ})}$ for the singular toric $n$-fold (of which $\PP_{\Delta}$ is a blow-up).}  Write $\DD_{\Delta}:=\PP_{\Delta}\setminus \mathbb{G}_m^n$ as before, $X_t \subset\PP_{\Delta}$ for the Zariski closure of $\{1=t\phi(\ux)\}$, and $Z:=\DD_{\Delta}\cap X_0$ for the base locus of the resulting pencil.

As in \cite[Thm. 4]{DvdK}, we may fix a sequence of blow-ups of $\PP_{\Delta}$ (typically along successive proper transforms of components of $Z$) with composition $\beta\colon \cx\to \PP_{\Delta}$, such that:
\begin{itemize}[leftmargin=0.5cm]
\item $\cx$ is smooth;
\item $\tfrac{1}{\phi(\ux)}$ extends to a holomorphic map $\pi\colon\cx\to \PP^1$;
\item $\cx\setminus \pi^{-1}(0)$ contains $\mathbb{G}_m^n$ as a Zariski open subset; and
\item $\text{dlog}(\ux):=\tfrac{dx_1}{x_1}\wedge\cdots\wedge\tfrac{dx_n}{x_n}$ extends to a holomorphic $n$-form on $\cx\setminus\pi^{-1}(0)$.
\end{itemize}
We shall assume that $\beta$ may be chosen in such a way that this extended form is nowhere vanishing, so that the $\tx_t:=\pi^{-1}(t)$ are Calabi-Yau for $t$ not in the discriminant locus $\Sigma$.  This is weaker than assuming $\phi$ ``generic'', and implies that $\beta_t:=\beta|_{\tx_t}\colon \tx_t\twoheadrightarrow X_t$ is a crepant resolution for $t\notin \Sigma$ (and also, that $\cx\setminus \tx_0$ is a log Calabi-Yau variety).  Despite the notation, $\tx_t$ is \emph{not} smooth for $t\in\Sigma$.

\begin{defn}\label{d31a}
(a) The \emph{compact LG-model} associated to $\phi$ is the family $\pi\colon \cx \to \PP^1_t$ of CY $(n-1)$-folds $\tx_t$ just constructed.  We may view its total space $\cx$ as a smooth compactification of the pencil $\{1=t\phi(\ux)\}\subset \mathbb{G}_m^n\times (\PP^1_t\setminus\{0\})$, and $\tx_0$ as a blow-up of $\DD_{\Delta}$.

(b) The (noncompact) \emph{LG-model} associated to $\phi$ is the restriction $\cx{\setminus} \tx_0 \to \PP_t^1{\setminus}\{0\}$ of $\pi$.
\end{defn}

\begin{example}\label{ex31a}
Here are some Laurent polynomials in 2 variables (for $n=3$, see \S\S\ref{SIVA}-\ref{SIVC}), together with the Kodaira types of the singular fibers of $\pi$ (first and last at $t=0$ resp. $\infty$):
\[\begin{tikzpicture}
\node [] at (-4.7,-0.5) {$i$};
\node [] at (-4.7,-2) {$1$};
\node [] at (-4.7,-4) {$2$};
\node [] at (-4.7,-6) {$3$};
\node [] at (-4.7,-8) {$4$};
\node [] at (-4.7,-10) {$5$};
\node [] at (-4.7,-12) {$6$};
\node [] at (1,-0.5) {$\phi^{(i)}$};
\node [] at (-2.5,-0.5) {$\Delta^{(i)}$};
\node [] at (4.8,-0.5) {singular fibers};
\node [] at (4.8,-2) {$\I_9$, $\I_1$, $\I_1$, $\I_1$(, $\I_0$)};
\node [] at (4.8,-4) {$\I_8$, $\I_1$, $\I_1$, $\I\I$};
\node [] at (4.8,-6) {$\I_5$, $\I_1$, $\I_1$, $\I_5$};
\node [] at (4.8,-8) {$\I_6$, $\underset{5\text{ times}}{\underbrace{\I_1,\ldots,\I_1}}$, $\I_1$};
\node [] at (4.8,-10) {$\I_3$, $\underset{9\text{ times}}{\underbrace{\I_1,\ldots,\I_1}}$(, $\I_0$)};
\node [] at (4.8,-12) {$\I_3$, $\I_1$, $\I\mathrm{V}^*$};
\draw [gray] (-5,-1) -- (7,-1);
\draw [gray] (-4.3,0) -- (-4.3,-12.5);
\draw [gray] (2.8,0) -- (2.8,-12.5);
\draw [gray] (-0.7,0) -- (-0.7,-12.5);
\node [] at (1,-2) {$x+y+\tfrac{1}{xy}$};
\node [align=center] at (1,-4) {$16x+y-3xy$\\$-6+\tfrac{1}{xy}$};
\node [] at (1,-6) {$\frac{(1-x)(1-y)(1-x-y)}{xy}$};
\node [align=center] at (1,-8) {{$\tfrac{x^2}{y}+\tfrac{y}{x}+\tfrac{1}{x y}$} \\{$-4^{\frac{1}{3}}3^{\frac{1}{2}}$}};
\node [align=center] at (1,-10) {$\tfrac{1}{xy}F_3$ \\{\tiny ($F_3$ a general cubic)}};
\node [] at (1,-12) {$\frac{(1+x+y)^3}{xy}$};
\filldraw [black,thick,fill=cyan] (-1.6,-2) -- (-2.5,-1.1) -- (-3.4,-2.9) -- (-1.6,-2);
\filldraw [black,thick,fill=cyan] (-1.6,-3.1) -- (-1.6,-4) -- (-3.4,-4.9) -- (-2.5,-3.1) -- (-1.6,-3.1);
\filldraw [black,thick,fill=cyan] (-2.5,-5.1) -- (-1.6,-6) -- (-1.6,-6.9) -- (-3.4,-6.9) -- (-3.4,-5.1) -- (-2.5,-5.1);
\filldraw [black,thick,fill=cyan] (-4,-7) -- (-4,-9) -- (-1,-9) -- (-4,-7);
\filldraw [black,thick,fill=cyan] (-4,-9.5) -- (-1,-12.5) -- (-4,-12.5) -- (-4,-9.5);
\draw [blue] (-1.6,-2) -- (-3.4,-2);
\draw [blue] (-2.5,-1.1) -- (-2.5,-2.9);
\draw [blue] (-1.6,-4) -- (-3.4,-4);
\draw [blue] (-2.5,-3.1) -- (-2.5,-4.9);
\draw [blue] (-1.6,-6) -- (-3.4,-6);
\draw [blue] (-2.5,-5.1) -- (-2.5,-6.9);
\draw [blue] (-1,-8) -- (-4,-8);
\draw [blue] (-3,-7) -- (-3,-9);
\draw [blue] (-1,-11.5) -- (-4,-11.5);
\draw [blue] (-3,-9.5) -- (-3,-12.5);
\end{tikzpicture}\]
For instance, the last two share the same $\PP_{\Delta}\cong \PP^2$, but have different $\cx$'s: obtained by blowing up at 9 distinct points (for the general cubic), vs. blowing up three times at each of three points.
\end{example}

\subsection{Variation of Hodge structure}\label{SIIB}

On a neighborhood of $t=0$, consider the family of vanishing $(n-1)$-cycles $\gamma_t$ on $\tx_t$ whose image under $\text{Tube}\colon H_{n-1}(\tx_t,\ZZ)\to H_n(\cx\setminus \tx_t,\ZZ)$ is $[\beta^{-1}(\mathbb{T}^n)]$, where $\mathbb{T}^n:=\cap_{i=1}^n \{|x_i|=1\}$.
The family of holomorphic forms
\begin{equation}\label{e32a}
\omega_t := \frac{1}{(2\pi\ay)^{n-1}}\mathrm{Res}_{\tx_t} \left( \frac{\mathrm{dlog}(\ux)}{1-t\phi(\ux)} \right)	\in \Omega^{n-1}(\tx_t)
\end{equation}
then has the holomorphic period
\begin{equation}\label{e32b}
A(t):=\int_{\gamma_t}\omega_t = \frac{1}{(2\pi\ay)^n}\oint_{\mathbb{T}^n}\frac{\mathrm{dlog}(\ux)}{1-t\phi(\ux)}=\sum_{k\geq 0}a_k t^k,\end{equation}
where $a_k=[\phi^k]_{\uo}$ are the constant terms in powers of $\phi$.

Writing $\pi_{\cu}\colon \cx_{\cu}\to \cu$ for the restriction of $\pi$ over $\cu:=\PP^1\setminus \Sigma$, the local system $\HH^{n-1}:=R^{n-1}(\pi_{\cu})_*\QQ$ has maximal unipotent monodromy\footnote{In this paper, a unipotent monodromy operator $T=e^N$ is \emph{maximally unipotent} if $N^{n-1}\neq 0$.} at $t=0$.  It underlies a (polarized) VHS with sheaf of holomorphic sections $\ch^{n-1}\cong\HH^{n-1}\otimes \co_{\cu}$ and Gauss-Manin connection $\nabla$.  In fact, we will work with the sub-local-system $\HH_v^{n-1}$ orthogonal to the fixed part $\HH_f^{n-1}=H^0(\cu,\HH^{n-1})$.  The corresponding sub-VHS $\ch_v^{n-1}\subseteq \ch^{n-1}$ contains the Hodge line $\ch^{n-1,0}= (\pi_{\cu})_* \Omega^{n-1}_{\cx_{\cu}}$.  On the level of $d_{\pi_{\cu}}$-closed-form representatives, the polarization $\langle\,,\,\rangle\colon \ch^{n-1}_v\times \ch^{n-1}_v\to \co$ is simply given by $\langle \omega,\eta\rangle=\int_{\tilde{X}_t}\omega_t \wedge \eta_t$.  

For simplicity, we shall henceforth \emph{assume} that $\ch^{n-1}_v$ is irreducible, not just as a $\QQ$-VHS but as a $\mathcal{D}$-module (or $\CC$-VHS).\footnote{When $n=3$, for instance, this assumption rules out (finite) monodromy in $\ch^2_{\text{alg}}$; for $n=2$, it is vacuous.}  Let $L\in \CC[t,\delta_t]$ be the differential operator with $(\ch_v^{n-1},\nabla)\cong \mathcal{D}/\mathcal{D}L$, of degree $d$ and order $r=\mathrm{rk}(\HH^{n-1}_v)$, normalized so that the coefficient of $\delta_t^r$ is $1$ at $t=0$.

A putative mirror to the LG-model is given by the following folklore

\begin{conj}
 \label{c32a}
There is a weak Fano $n$-fold $(\mathrm{X}^{\circ},\omega)$, determined by the triple $(\PP_{\Delta},\DD_{\Delta},Z)$ and admitting a toric degeneration to $\PP'_{\Delta^{\circ}}$, from which one may recover $\ch^{n-1}_v$.  \textup{(}In particular, for generic $\phi$, we have $\mathrm{X}^{\circ}=\PP_{\Delta^{\circ}}$.\textup{)}
\end{conj}

\noindent Conversely, it is hoped that by studying ``special'' Laurent polynomials and classifying the associated local systems, one obtains a classification of weak Fano varieties admitting a toric degeneration.  While Conjecture \ref{c32a} is vague as stated, it will be refined below:  a mechanism for recovering $\ch^{n-1}_v$ (in some cases) is given in Conjecture \ref{c33a} \textit{bis}; while the Hodge-theoretic sense in which $\mathrm{X}^{\circ}$ is mirror to $\cx\setminus \tx_0\to\aff^1$ is the subject of Conjecture \ref{c41a}.

\subsection{Quantum $\mathcal{D}$-module}\label{SIIC}

For simplicity, we assume in this subsection that the Picard rank $\rho(\mathrm{X}^{\circ})=1$ (and $\mathrm{X}^{\circ}$ is Fano).  In the standard way \cite{Gv07}, one uses genus-zero Gromov-Witten theory to construct a quantum product ``$\star$'' on $H^*(\mathrm{X}^{\circ})\otimes \CC[s^{\pm 1}]$.  This is endowed with a $\mathcal{D}$-module structure by letting $\delta_s$ act via $1\otimes \delta_s -(K_{\mathrm{X}^{\circ}}\star)\otimes 1$.

Now let $\ch_v^{n-1}$ be as in Conjecture \ref{c32a}, and $L$ be the corresponding Picard-Fuchs equation, with Fourier-Laplace transform $\hat{L}$.  (Here we recall that the FL-transform and its inverse are given on functions/solutions\footnote{If $f(t)=\sum_k c_k t^k$ is a power-series, this gives $\hat{f}(s)=\sum_k \tfrac{c_k}{k!} s^k$.} by
\begin{equation}\label{e33a}
\hat{f}(s):=\frac{1}{2\pi\ay}\oint f(t)e^{{s}/{t}}\frac{dt}{t}	\;\;\;\;\text{and}\;\;\;\; \check{F}(t):=\frac{1}{t}\int_{0}^{\infty}F(s)e^{-s/t}ds,
\end{equation}
and on operators by replacing $\partial_t \leftrightarrow -s$ and $t\leftrightarrow \partial_s$.)
Then we have the following amplification of that Conjecture when $\omega=-K_{\mathrm{X}^{\circ}}$:\addtocounter{thm}{-1}
\begin{conjec}{\textbf{\textit{bis.}}}\label{c33a}
As $\mathcal{D}$-modules, $H^*(\mathrm{X}^{\circ})\otimes \CC[t^{\pm 1}] \cong \mathcal{D}/\mathcal{D}\hat{L}$.
\end{conjec}
\noindent That is, by applying inverse FL-transform to $\hat{L}$ we should obtain $\ch^{n-1}_v$.  

\begin{rem}
Conjecture \ref{c33a} \textit{bis} still makes sense without the assumption that $\rho(\mathrm{X}^{\circ})=1$, by pulling back the quantum $\mathcal{D}$-module via the anticanonical co-character of the N\'eron-Severi torus.
It has been checked for an enormous number of Fano varieties of dimension 2 and 3 \cite{CCGGK}. A more recent \emph{general} result affirms it for Fano varieties admitting a toric degeneration with terminal singularities (and preserving $H_2$), in which case the nonzero integer points of $\Delta$ are vertices and $\phi=\sum_{\um\in\Delta\cap(\ZZ^n\setminus\{\uo\})}\ux^{\um}$ \cite[Cor.~4.7]{BGM}.
\end{rem}

Assuming Conjecture \ref{c33a} \textit{bis} holds for $\mathrm{X}^{\circ}$, we write $(\hat{L}=)\sum_{i,j}\beta_{ij}t^i\delta_t^j$ for the (irregular) differential operator killing the generator $1\otimes 1$ of the quantum $\mathcal{D}$-module, and convert this into an (irregular) \emph{quantum recursion}
\begin{equation}\label{e33b}
\hat{R}\colon\;\;\;\sum_{i,j}\beta_{ij}(k-i)^j \hat{u}_{k-i}=0\;\;(\forall k)	
\end{equation}
by applying $\hat{L}$ to a power series $\sum_k \hat{u}_k s^k$.  We consider a basis of solutions $\{\hat{u}^{(i)}_.\}_{i=0}^{d-1}$, defined over the same field as $L$ (typically $\QQ$), with $\hat{u}^{(i)}_k=0$ for $k<i$ and $\hat{u}^{(i)}_i=\tfrac{1}{i!}$.  Regularizing via the inverse FL transform, $\hat{L},\hat{R}$ become $L,R$, with solutions $u_k = k!\hat{u}_k$; in particular, we have $u^{(i)}_k=0$ for $k<i$ and $u^{(i)}_i=1$.

We shall take the basis to be chosen so that $u^{(0)}_k=a_k$ is as in \eqref{e32b}, and impose one more assumption:  that the $\{a_k\}$ are nonzero.  There are various ways to further normalize $u^{(1)}_.,\ldots,u^{(d-2)}_.$.  For instance, there are $r_0:=\mathrm{rk}((\psi_0\ch^{n-1}_v)^{T_0})\leq d$ independent holomorphic solutions to $L(\cdot)=0$  at the origin, which we take to be given by the generating series of $u^{(0)}_.,\ldots, u^{(r_0-1)}_.$.\footnote{If $(\psi_0 \ch^{n-1}_v)^{T_0}$ is Hodge-Tate, with $r_0$ distinct graded pieces $\{\QQ(-p_i)\}_{i=0}^{r_0-1}$ (with $p_0=0$), one can take $\sum_k u^{(i)}_k t^k$ ($i=0,\ldots, r_0-1$) to be the $\CC$-periods of $\omega$, against local sections $\varphi^{(i)}$ of $\HH^{n-1}_{v,\CC}$ passing through $\CC(-p_i)$ at $t=0$.} The remaining $d-r_0$ generating series will then be solutions to inhomogeneous equations $L(\cdot)=g_i(t)\in \CC[t]$.\footnote{Writing $L=\sum_{\ell=0}^d t^{\ell}P_{\ell}(\delta_t)$, it is reasonable to expect that $P_0(T)=\prod_{i=0}^{r_0-1}(T-i)^{n-2p_i}$, and then we may assume that $\sum_k u_k^{(i)} t^k$ ($i=r_0,\ldots, d-1$) solves $L(\cdot)=P_0(i)t^i$.} In particular, it will be important in \S\ref{SIV} that when $d=2$ and $r=n$ ($\implies P_0(T)=T^n$), $\sum_{k\geq 1}u^{(1)}_k t^k$ solves $L(\cdot)=t$.

Slightly generalizing the definition in \cite{Gv}, we propose

\begin{defn}\label{d33a}
The \emph{Ap\'ery constants} of $\mathrm{X}^{\circ}$ are the limits
\begin{equation}\label{e33c}
\alpha_{\mathrm{X}^{\circ}}^{(i)}:=\lim_{k\to \infty}\frac{\hat{u}^{(i)}_k}{\hat{u}^{(0)}_k}=\lim_{k\to \infty}\frac{u^{(i)}_k}{u^{(0)}_k},
\end{equation}
for $1\leq i\leq d-1$.  (When $d=2$, we simply write $\alpha_{\mathrm{X}^{\circ}}$.)
\end{defn}

\begin{rem}\label{r33a}
The closely related definition in \cite{Gk} (of an Ap\'ery class $A(\mathrm{X}^{\circ})\in H_{\text{prim}}^*(\mathrm{X}^{\circ})$, with the constants appearing as its coefficients) only considers the first $r_0-1$ Ap\'ery constants, corresponding to solutions of the homogeneous equation.  (As was first realized by Galkin and Iritani in the case of Grassmannians, these should correspond to the restriction of the regularized gamma class of $\mathrm{X}^{\circ}$ to the Lefschetz coprimitive part of cohomology.)

However, the focus in [op. cit.] is on large-dimensional examples for which $r_0=d$; taking hyperplane sections preserves $d$ as well as the $\alpha^{(i)}$ (in our sense), even as $r_0$ decreases.  Since four of the five $3$-dimensional examples we consider in \S\ref{SIV} are indeed obtained as multisections of homogeneous Fano varieties with $(\dim(H^*_{\text{prim}}(\mathrm{X}^{\circ})=)\,r_0=2=d$, the ``inhomogeneous'' Ap\'ery constants for the $3$-folds in \cite{Gv} are connected to the constants in \cite{Gk} in this way (albeit with a slightly different normalization).
\end{rem}

\begin{rem}\label{r33b}
Adding a constant $c$ to $\phi$ conjugates $\hat{L}$ by $e^{cs}$, which does not affect the Ap\'ery constants. We may thus choose the constant term to make $\phi=0$ (i.e. $\tilde{X}_{\infty}$) singular.
\end{rem}

By ``specializing'' Laurent polynomials, we hope not just to classify Fanos but to arrive at a B-model, Hodge-theoretic interpretation of their Ap\'ery numbers.  But there is a new twist.  Consider the simplest case, where $d=2$ and $r_0=1$, and write $u_k^{(0)}=a_k$, $u_k^{(1)}=:b_k=0,1,\ldots$, and $\alpha_{\mathrm{X}^{\circ}}=\lim_{k\to \infty}\tfrac{b_k}{a_k}$.  While $A(t)=\sum_{k\geq 0}a_k t^k$ is just the holomorphic period, the $\{b_k\}$ and hence $\alpha_{\mathrm{X}^{\circ}}$ are \emph{not} visible from $\ch^{n-1}_v$ alone.  It is for this reason that we turn to variations of MHS in the next section.

\section{Higher normal functions}\label{SIII}

\subsection{Variation of mixed Hodge structure}\label{SIIIA}

Fix a Laurent polynomial $\phi$ subject to the assumptions in $\S$\ref{SIIA}, and write $X_t^*=X_t\cap\mathbb{G}_m^n$ for the level sets of $\tfrac{1}{\phi}$.

\begin{prop}\label{p41a}
Suppose that $\phi$ is ``generic'' in the sense that it is $\Delta$-regular.\footnote{That is, the intersections of $X_t$ ($t\neq 0$) with each of the torus-orbits in $\DD_{\Delta}$ are smooth and reduced; this is equivalent to the meaning of genericity in \S\ref{SI}.}
Then as MHSs, we have $H^n(\cx\setminus \tx_0,\tx_t)\cong H^n(\mathbb{G}_m^n,X_t^*)$ and $H^n(\cx\setminus \tx_t,\tx_0)\cong H^n(\PP_{\Delta}\setminus X_t,\DD_{\Delta}\setminus Z)$ for $t\neq 0$.
\end{prop}
\begin{proof}
With the additional genericity assumption, the construction of $\cx$ in \S\ref{SIIA} proceeds by blowing up $\PP_{\Delta}$ \emph{once} along each component of $Z$ (or rather, their successive strict transforms).  An easy local computation shows that the restriction of $\pi$ to the exceptional divisor $\cE\subset\cx$ of $\beta$ is then locally constant over $\PP^1\setminus\{0\}$.  In particular, writing $\cE_t:=\cE\cap\tx_t$, $\cE_0$ is a deformation retract of $\cE\setminus\cE_t$ for any $t\neq 0$.

Since $\mathbb{G}_m^n =\cx\setminus (\tx_0 \cup \cE)$ and $X_t^*=\tx_t\setminus\cE_t$, we have 
$$
H^n(\mathbb{G}_m^n,X_t^*)\cong H^n(\cx\setminus(\tx_0\cup\cE),\tx_t\setminus\cE_t) \cong H_n(\cx\setminus \tx_t,\tx_0\cup(\cE\setminus \cE_t))(-n)
$$
and $H^n(\cx\setminus\tx_0,\tx_t)\cong H_n(\cx\setminus\tx_t,\tx_0)(-n)$, which fit together in the long-exact sequence
\begin{equation}\label{e4.1s}
\to H_n(\cE\setminus\cE_t,\cE_0)(-n)\to H^n(\cx\setminus\tx_0,\tx_t)\to  H^n(\mathbb{G}_m^n,X_t^*)\to H_{n-1}(\cE\setminus\cE_t,\cE_0)(-n)\to
\end{equation}
whose end terms are zero by the deformation retract property.  The other isomorphism follows by Lefschetz duality.
\end{proof}

\noindent The isomorphisms in Proposition \ref{p41a} typically \emph{fail} without the genericity assumption on $\phi$; that is, the restriction morphism $H^n(\cx\setminus \tx_0,\tx_t)\to H^n(\mathbb{G}_m,X_t^*)$ is not an isomorphism, and the left-hand object better reflects the topology of the LG-model.\footnote{For example, $H^2(\mathbb{G}_m^2,X_t^*)$ does not distinguish between cases $i=1$ and $i=6$ in Example \ref{ex31a}/\ref{ex44a}, and for the non-generic case $i=6$ does not agree with $\cv_{\phi,t}$.}  So from here on we shall focus \emph{primarily} on the VMHS 
\begin{equation}\label{e41a}
\cv_{\phi,t}:= H^n(\cx\setminus \tx_0,\tx_t)
\end{equation}
over $\cu:=\PP^1\setminus\Sigma$, and its dual
\begin{equation}\label{e41b}
\cv_{\phi,t}^{\vee}(-n)\cong H^n(\cx\setminus \tx_t,\tx_0).
\end{equation}
Of course, \emph{if} $\phi$ is generic, by Prop.~\ref{p41a} and Thm.~\ref{t23a} the right-hand term of \eqref{e41b} is nothing but a restriction of $\hat{\tau}^{\Delta}_{\text{GKZ}}$.  This suggests the following generalization of Theorem \ref{t23b}:

\begin{conj}[\cite{KKP}]\label{c41a}
For $t\notin \Sigma$, we have for each $k$ 
\begin{equation}\label{e41c}
\mathrm{rk}(\gr_F^{n-k}\cv^{\vee}_{\phi} (-n))=\dim (H^{k,k}(\xc)).
\end{equation}
\end{conj}
\begin{rem}\label{r41a}
(i) The form in which we state this conjecture combines \cite[Conj.~3.7]{KKP} with \cite[Thm.~3.1]{Ha}.

(ii) By Serre duality, \eqref{e41c} would imply that $\mathrm{rk}(\gr_F^{n-k}\cv_{\phi})=\mathrm{rk}(\gr_F^k\cv_{\phi})$ for all of our LG-models, which has been proved by Harder \cite[Cor.~2.2.7]{Ha2}.

(iii) There is a plausible amplification of \eqref{e41c} to an equality of $\CC$-MHSs with nilpotent endomorphism:  on the left-hand (B-model) side, the LMHS $(\psi_t H^n(\cx\setminus \tx_0,\tx_t),F_{\lim}^{\bullet},M_{\bullet},N_0)$ at $t=0$;\footnote{Here $M_{\bullet}$ is the monodromy weight filtration of $N_0$ relative to $W_{\bullet}$, the weight filtration on the VMHS $H^n(\cx\setminus\tx_0,\tx_t)$; if this extended form of the conjecture holds then $M_{\bullet}$ is just $W(N_0)[-n]_{\bullet}$. Moreover, if $\mathrm{X}^{\circ}$ is Fano (not just weak Fano), then it implies that this LMHS is Hodge-Tate.} on the right-hand (A-model) side, $\oplus_k H^{k,k}(\mathrm{X}^{\circ})$ endowed with the ``downward'' Hodge filtration $F^{p}=\oplus_{k\leq n-p}H^{k,k}(\mathrm{X}^{\circ})$, together with $N=\cup[-K_{\mathrm{X}^{\circ}}]$ and its weight filtration $W(N)[-n]_{\bullet}$ (centered about $n$).
\end{rem}

\noindent In order to relate limits of extension classes in \eqref{e41a}-\eqref{e41b} to Ap\'ery constants of $\xc$, we shall need to kill off intermediate extensions which would otherwise ``obstruct'' these classes.  This will be accomplished by placing a ``$K$-theoretic'' constraint on the Laurent polynomial:

\begin{defn}\label{d41a}
We say $\phi$ is \emph{tempered} if the coordinate symbol $\{x_1,\ldots,x_n\}\in H_{\mathcal{M}}^n(\mathbb{G}_m^n,\QQ(n))$ lifts to a class in $H_{\mathcal{M}}^n(\cx\setminus\tx_0,\QQ(n))$.\end{defn}

\noindent Henceforth we shall mainly be concerned with the case where $\phi$ is tempered.  When $n=2$, this is just the condition that the edge polynomials of $\phi$ be cyclotomic \cite{mmm}; some methods for checking temperedness for $n=3,4$ are given in \cite[$\S$3]{DK1}.  Up to scale, tempered reflexive Laurent polynomials are defined over $\bar{\QQ}$ [op. cit., Prop. 4.16] and are thereby rigid.

\subsection{Admissible and geometric normal functions}\label{SIIIB}

A reference for the material that follows is \cite[$\S$2.11-12]{KP1}.

\begin{defn}\label{d42a}
A \emph{higher normal function} on $\cu$ is (equivalently)
\begin{enumerate}[(i),leftmargin=0.7cm]
\item a VMHS of the form $0\to\ch\overset{\imath}{\to}\cv\to\QQ(0)\to 0$, or
\item a holomorphic, horizontal section $\nu$ of $J(\ch):=\ch/(F^0\ch+\HH)$,
\end{enumerate}
where $\ch$ is a polarizable VHS of pure weight $-r<-1$.
\end{defn}
\noindent Here \emph{horizontal} means that, for each local holomorphic lift $\tilde{\nu}$ to $\ch$, we have $\nabla \tilde{\nu}\in F^{-1}\ch$.  For instance, given (i) we may locally lift $1\in \QQ(0)$ to $\nu_F \in F^0\cv$ and $\nu_{\QQ}\in\VV$ (the local system underlying $\cv$), then locally define $\tilde{\nu}$ (hence $\nu$ as in (ii)) by $\imath(\tilde{\nu})=\nu_{\QQ}-\nu_F$.

Let $\cv_e$ denote Deligne's canonical extension of $\cv$ to $\PP^1$.  Fixing disks $D_{\sigma}\subset\PP^1$ at each $\sigma\in\Sigma\,(=\PP^1\setminus \cu)$, with coordinate $t_{\sigma}$, we write $T_{\sigma}=e^{N_{\sigma}}T_{\sigma}^{\text{ss}}$ for the monodromy of $\VV$ on $D_{\sigma}^*=D_{\sigma}\setminus\{\sigma\}$, and $M^{\sigma}_{\bullet}$ for the monodromy-weight filtration of the LMHS $\psi_{\sigma}\ch$.  Suppose now that there exist ``lifts of $1$'':
\begin{itemize}[leftmargin=0.5cm]
\item $\nu_F^{\sigma}\in \Gamma(D_{\sigma},\cv_e)$ -- holomorphic, single-valued, with $\nu_F^{\sigma}|_{D_{\sigma}^*}$ in $F^0\cv$
\item $\nu_{\QQ}^{\sigma}\in \Gamma(\widetilde{D^*_{\sigma}}^{\text{un}},\VV)^{T_{\sigma}^{\text{ss}}}$ -- flat, multivalued, with $N_{\sigma}\nu_{\QQ}^{\sigma}\in M_{-2}^{\sigma}\psi_{\sigma}\ch$.
\end{itemize}
Then we may confer on $\cv_e|_{\sigma}$ the status of a MHS $\psi_{\sigma}\cv$ as follows:
\begin{itemize}[leftmargin=0.5cm]
\item the weight filtration $M^{\sigma}_{\bullet}$ extends that on $\psi_{\sigma}\ch$, adding $\nu_{\QQ}^{\sigma}$ to $M^{\sigma}_0$;
\item the Hodge filtration $F^{\bullet}_{\sigma}$ extends that on $\psi_{\sigma}\ch$, adding $\nu_F^{\sigma}(\sigma)$ to $F^0_{\sigma}$;
\item the $\QQ$-structure $(\psi_{\sigma}\cv)_{\QQ}$ is easiest to describe after a base-change (to kill off $T^{\text{ss}}_{\sigma}$), as the specialization of $\exp(-\tfrac{\log(t_{\sigma})}{2\pi\ay}N_{\sigma})\VV \subset\cv_e$ at $\sigma$.
\end{itemize}

\begin{defn}\label{d42b}
The HNF $\nu$ is \emph{admissible}, written $\nu\in \mathrm{ANF}(\ch)$, if this LMHS $\psi_{\sigma}\cv$ (equivalently, $\nu_F^{\sigma}$ and $\nu_{\QQ}^{\sigma}$) exists at each $\sigma\in\Sigma$.  If, in addition, we may choose $\nu_{\QQ}^{\sigma}$ so that $N_{\sigma}\nu^{\sigma}_{\QQ}=0$, then the \emph{limit} $\lim_{\sigma}\nu\in J((\psi_{\sigma}\ch)^{T_{\sigma}})$ is defined;\footnote{$\lim_{\sigma}\nu$ is given by $\imath(\widetilde{\lim_{\sigma}\nu})=\nu_{\QQ}^{\sigma}-\nu_F^{\sigma}(0)$, as in the passage from (i) to (ii) above.} otherwise, $\nu$ is \emph{singular} at $\sigma$.
\end{defn}

\begin{rem}\label{r42a}
Writing $\ch=\ch_f\oplus\ch_v$ for the decomposition into fixed and variable parts (with $\ch_f = H_f \otimes \co_{\cu}$), we claim that
\begin{equation}\label{e42}
0\to J(H_f) \to \mathrm{ANF}(\ch) \to \mathit{Hg}(H^1(\cu,\ch))\to 0
\end{equation}
is exact.  Indeed, since $\mathrm{ANF}(\ch)\cong \mathrm{Ext}^1_{\mathrm{AVMHS}(\cu)}(\QQ(0),\ch)$, this follows at once from the spectral sequence $$R\mathrm{Hom}_{\text{MHS}}(\QQ(0),-)\circ R\Gamma_{\cu}\implies R\mathrm{Hom}_{\text{AVMHS}(\cu)}(\QQ(0),-)$$ and triviality of $\mathrm{Ext}_{\text{MHS}}^{i>1}$, by using the identifications $H^0(\ch)=H_f$, $J(H_f)\cong \mathrm{Ext}^1_{\text{MHS}}(\QQ(0),H_f)$, and $\mathit{Hg}(H^1(\ch))\cong \mathrm{Hom}_{\text{MHS}}(\QQ(0),H^1(\ch))$.
\end{rem}

\noindent We say that $\nu\in\mathrm{ANF}(\ch^{2p-r}(p))$ is of \emph{geometric origin} when it arises from a motivic cohomology class in 
\begin{equation}\label{e42a}
H^{2p-r+1}_{\mathcal{M}}(\cx_{\cu},\QQ(p))\cong \gr_{\gamma}^p K^{\text{alg}}_{r-1}(\cx_{\cu})_{\QQ}\cong \mathrm{CH}^p(\cx_{\cu},r-1).
\end{equation}
The most convenient representatives are found in the right-hand term, the \emph{higher Chow groups} of Bloch \cite{Bl,Bl2}, which (in their cubical formulation) are defined as the $(r-1)^{\text{st}}$ homology of a complex $\left(Z^p(X,\bullet),\partial\right)$ of codim.-$p$ cycles on $X\times \square^{\bullet}$, where $\square:=\PP^1\setminus\{1\}$.\footnote{Elements of $Z^p(X,n)$ must meet all faces $X\times \square^m$ (defined by setting $\square$-coordinates to $0$ or $\infty$) properly, and $\partial$ is given by an alternating sum of intersections with codim.-$1$ faces.  See \cite[$\S$1]{DK1} for a brief introduction to higher cycles and their Hodge-theoretic invariants. } Given a cycle $\mathcal{Z}$ in \eqref{e42a}, its restrictions $\mathcal{Z}_t\in\mathrm{CH}^p(\tx_t,r-1)$ have (for each $t\in\cu$) Abel-Jacobi/regulator invariants\footnote{These may be computed as the class of a closed $(2p-r-1)$-current $(2\pi\ay)^p \delta_{\Gamma}+(2\pi\ay){p-r+1}(\mathcal{Z}_t)_* R_{r-1}$, where $R_{r-1}$ is a standard $(r-2)$-current on $\square^{r-1}$, and $\Gamma$ is a chain bounding on $(\mathcal{Z}_t)_* \RR_{<0}^{r-1}$ [loc. cit.].  One defines these \emph{regulator currents} inductively by $$R_{\ell}(x_1,\ldots,x_{\ell}):=\log(x_1)\tfrac{dx_1}{x_1}\wedge\cdots\wedge \tfrac{dx_{\ell}}{x_{\ell}}-2\pi\ay \delta_{T_{x_1}}\cdot R_{\ell-1}(x_2,\ldots,x_{\ell}),$$where $T_x:=x^{-1}(\RR_{<0})$; they satisfy $$d[R_{\ell}]=\mathrm{dlog}(\ux)-(2\pi\ay)^{\ell}\delta_{\cap^{\ell}_{i=1} T_{x_i}}+\textstyle{\sum_{i=1}^{\ell}}(-1)^i R_{\ell-1}(x_1,\ldots,\widehat{x_i},\ldots,x_{\ell})\delta_{(x_i)}.$$} 
\begin{equation}\label{e42b}
\mathrm{AJ}(\mathcal{Z}_t)=:\nu_{\mathcal{Z}}(t)\in J(H^{2p-r}(\tx_t)(p)).
\end{equation}
By \cite[Thm. 7.3]{BZ}, these glue together into an admissible normal function, so that $\mathcal{Z}\mapsto \nu_{\mathcal{Z}}$ defines a map
\begin{equation}\label{e42c}
\mathrm{AJ}_{\phi}\colon \mathrm{CH}^p(\cx_{\cu},r-1)\to \mathrm{ANF}(\ch^{2p-r}(p)). \end{equation}
Composing with projection to $\mathrm{ANF}(\ch_v^{2p-r}(p))\cong \mathit{Hg}(H^1(\cu,\ch_v^{2p-r}(p)))$ (cf. Remark \ref{r42a}) defines $\mathrm{AJ}^v_{\phi}$ and ${\nu}^v_{\mathcal{Z}}$, for which we have the following special case of the Beilinson-Hodge Conjecture:
\begin{conj}[BHC]\label{c42a}
For $\cx_{\cu}$ defined over $\bar{\QQ}$, $\mathrm{AJ}^v_{\phi}$ is surjective.  That is, admissible and geometric HNFs with values in $\ch_v^{2p-r}(p)$ are the same thing.
\end{conj}
\noindent The equivalence in Definition \ref{d42a}, as well as the notion of admissibility, persist with $\ch$ merely a VMHS; we shall loosely refer to such VMHSs (and the corresponding sections of $J(\ch)$) in this more general setting as \emph{mixed HNFs}.

\subsection{$\cv_{\phi}$ as a (mixed) higher normal function}\label{SIIIC}

First we set (dually)
\begin{equation}\label{e43a}
\left\{
\begin{split}
\overline{H}^{\ell}(X_t^*)&:=\text{coker}\{H^{\ell}(\mathbb{G}^m_n)\to H^{\ell}(X_t^*)\}\\
\underline{H}_{\ell}(X_t^*)&:=\ker\{H_{\ell}(X_t^*)\to H_{\ell}(\mathbb{G}_m^n)\}
\end{split}\right. 
\end{equation}
and write $\cv_{\phi,t}^*:=H^n(\mathbb{G}_m^n,X_t^*).$
Since $X_t^*$ is affine and $H^n(\mathbb{G}_m^n)\cong \QQ(-n)$, the relative cohomology sequence of the pair $(\mathbb{G}_m^n,X_t^*)$ yields an exact sequence of MHS
\begin{equation}\label{e43b}
0\to \overline{H}^{n-1}(X_t^*)\to\cv_{\phi,t}^*\overset{\Xi}{\to} \QQ(-n)\to 0.
\end{equation}
The kernel and cokernel of the restriction map
$$
\theta\colon \cv_{\phi,t}\to \cv_{\phi,t}^*
$$
may be read off from the end terms of \eqref{e4.1s}. Since we are not assuming that $\phi$ is $\Delta$-regular, these may no longer be zero, but we claim they are ``mild'' in the sense of having weights between $n$ and $2n-2$.  More precisely:
\begin{lem}\label{lnew}
The map $\theta$ induces isomorphisms on $\gr^W_{2n}(\cong \QQ(-n))$, $\gr^W_{2n-1}(\cong \{0\})$, and $\gr^W_{n-1}$ \textup{(}namely, $\ch_{v}^{n-1}\cong W_{n-1}\overline{H}^{n-1}(X_t^*)=W_{n-1}H^{n-1}(X_t^*)\textup{)}$.
\end{lem}
\begin{proof}
Clearly $\gr^W_{2n}\cv_{\phi,t}^*\cong \QQ(-n)$ and $\gr^W_{2n-1}\cv_{\phi,t}^*\cong\{0\}$.  Since $\mathrm{dlog}(\ux)$ extends to $\Omega^n(\cx\setminus \tx_0)$, the composition $\Xi\circ \theta\colon \cv_{\phi,t}\to H^n(\mathbb{G}_m^n)$ is surjective; and by Prop.~\ref{p44a} and Rem.~\ref{r44a} below, $\cv_{\phi,t}^{\circ}:=\ker(\Xi\circ \theta)$ has weights in $[n-1,2n-2]$.  Notice that we can write \eqref{e41a} as an extension
\begin{equation}\label{eext}
0\to \cv_{\phi,t}^{\circ}\to \cv_{\phi,t}\overset{\Xi\circ \theta}{\to} \QQ(-n)\to 0
\end{equation}
in analogy to \eqref{e43b}.

For $\gr^W_{n-1}$, we have $W_{n-1}\cv_{\phi,t}^*=W_{n-1}\overline{H}^{n-1}(X_t^*)=W_{n-1}H^{n-1}(X_t^*)$ by \eqref{e43a} and \eqref{e43b}, and $W_{n-1}\cv_{\phi,t}=\mathrm{coker}\{H^{n-1}(\cx\setminus \tx_0)\to H^{n-1}(\tx_t)\}=\ch_{v,t}^{n-1}$ by the global invariant cycle theorem. Clearly $\theta$ restricts to a map $\ch_{v,t}^{n-1}\to \overline{H}^{n-1}(X_t^*)$, which surjects onto the $W_{n-1}$ part by standard mixed Hodge theory, and is injective by the assumed irreducibility of $\ch_v^{n-1}$.
\end{proof}

Let $\Theta\colon \cv^{\circ}_{\phi,t}\to \overline{H}^{n-1}(X_t^*)$ be the map induced by $\theta$, so that $\eqref{eext} \mapsto \eqref{e43b}$ in 
\begin{equation}\label{bigT}
\Theta_*\colon \mathrm{Ext}^1_{\text{MHS}}(\QQ(-n),\cv^{\circ}_{\phi,t})\to \mathrm{Ext}^1_{\text{MHS}}(\QQ(-n),\overline{H}^{n-1}(X_t^*)).
\end{equation}
That is, we may view $\cv_{\phi}$ and $\cv_{\phi}^*$ as mixed HNFs related by $\Theta_*$, with Hodge-Deligne diagrams \emph{both} of the form on the left below:

\[\begin{tikzpicture}[scale=0.6]
\draw [<->,thick,gray] (0,5) -- (0,0) -- (5,0);
\draw [<->,thick,gray] (10,5) -- (10,0) -- (15,0);
\draw [<-left hook,thick,black] (5.5,2.5) -- (8.5,2.5);
\draw [thick,gray] (4,-0.2) -- (4,0.2);
\draw [thick,gray] (3,-0.2) -- (3,0.2);
\draw [thick,gray] (-0.2,3) -- (0.2,3);
\draw [thick,gray] (-0.2,4) -- (0.2,4);
\node [gray] at (3,-0.4) {\tiny $n-1$};
\node [gray] at (4,-0.4) {\tiny $n$};
\node [gray] at (-0.8,3) {\tiny $n-1$};
\node [gray] at (-0.4,4) {\tiny $n$};
\node [gray] at (5.3,0) {\small $p$};
\node [gray] at (0.4,5) {\small $q$};
\node [] at (7,3) {\small tempered};
\node [] at (7,2) {\small case};
\filldraw [color=blue, fill=blue!5, ultra thick, decorate, decoration={snake,amplitude=1pt}] (3,0) -- (3,3) -- (0,3) -- (3,0);
\filldraw [blue] (4,4) circle (3pt) node [anchor=south] {\small $\QQ(-n)$};
\filldraw [blue] (0,3) circle (3pt);
\filldraw [blue] (3,3) circle (3pt);
\filldraw [blue] (3,0) circle (3pt);
\draw [thick,gray] (14,-0.2) -- (14,0.2);
\draw [thick,gray] (13,-0.2) -- (13,0.2);
\draw [thick,gray] (9.8,3) -- (10.2,3);
\draw [thick,gray] (9.8,4) -- (10.2,4);
\node [gray] at (13,-0.4) {\tiny $n-1$};
\node [gray] at (14,-0.4) {\tiny $n$};
\node [gray] at (9.2,3) {\tiny $n-1$};
\node [gray] at (9.6,4) {\tiny $n$};
\node [gray] at (15.3,0) {\small $p$};
\node [gray] at (10.4,5) {\small $q$};
\draw [color=blue, ultra thick, decorate, decoration={snake,amplitude=1pt}] (13,0) -- (10,3);
\filldraw [blue] (14,4) circle (3pt) node [anchor=south] {\small $\QQ(-n)$};
\filldraw [blue] (10,3) circle (3pt);
\filldraw [blue] (13,0) circle (3pt);
\node [blue] at (14,2) {\small $W_{n-1}H^{n-1}(X_t^*)$};
\end{tikzpicture}\]
\begin{prop}\label{p43a}
$\cv_{\phi}^*$ is the \textup{(}mixed\textup{)} geometric HNF associated to the coordinate symbol $\{\ux\}=\{x_1,\ldots,x_n\}\in \mathrm{CH}^n(\mathbb{G}_m^n,n)$. Moreover, if $\phi$ is tempered, then $\cv_{\phi}$ and $\cv_{\phi}^*$ have a common \textup{(}pure\textup{)} geometric sub-HNF with Hodge-Deligne diagram of the form shown on the right, associated to the lift of $\{\ux\}$ to $\mathrm{CH}^n(\cx\setminus\tx_0,n)$ \textup{(}cf.~Definition \ref{d41a}\textup{)}.
\end{prop}
\begin{proof}
Each $\gamma\in \underline{H}_{n-1}(X_t^*,\QQ)$ may be written as $\partial\mu$ for a $n$-chain $\mu$ on $\mathbb{G}_m^n$.  The extension class $\nu_{\phi}(t)$ of \eqref{e43b} in $$J(\overline{H}^{n-1}(X_t^*)(n))\cong \mathrm{Hom}(\underline{H}_{n-1}(X_t^*,\QQ),\CC/\QQ(n))$$
is then computed on $\gamma$ (using Stokes's theorem) by $$\langle\tilde{\nu}_{\phi}(t),\gamma\rangle = \int_{\mu}\mathrm{dlog}(\ux)\underset{\QQ(n)}{\equiv}\int_{\mu}d[R_n]=\int_{\gamma}R_n|_{X_t^*}=\langle \mathrm{AJ}(\{\ux\}|_{X_t^*}),\gamma\rangle.$$
Therefore $\nu_{\phi}$ is the \emph{geometric} (mixed) HNF associated to the coordinate symbol $\{\ux\}=\{x_1,\ldots,x_n\}\in \mathrm{CH}^n(\mathbb{G}_m^n,n)$ (i.e., the graph of this $n$-tuple, viewed as a cycle in $\mathbb{G}_m^n\times\square^n$).

Now setting $H^n(\cx\setminus \tx_0)^{\circ}:= \ker\{H^n(\cx\setminus \tx_0)\to H^n(\tx_t)\}$, we can write $\cv_{\phi,t}$ as an extension in a manner different from \eqref{eext}, as
\begin{equation}\label{eext2}
0\to \ch_{v,t}^{n-1}\to \cv_{\phi,t} \to H^n(\cx\setminus \tx_0)^{\circ}\to 0.
\end{equation}
If $\phi$ is tempered, then $\{\ux\}$ lifts to $\xi \in \mathrm{CH}^n(\cx\setminus \tx_0,n)$, inducing a section of the MHS-morphism $ H^n(\cx\setminus \tx_0)^{\circ}\twoheadrightarrow H^n(\mathbb{G}_m^n)\cong \QQ(-n)$ under which \eqref{eext2} pulls back to an extension of the form
\begin{equation}\label{eext3}
0\to \ch_{v,t}^{n-1}\to \cv_{\xi,t}\to \QQ(-n)\to 0.
\end{equation}
The resulting morphism of sequences from \eqref{eext3} to \eqref{eext} is a partial splitting as shown in the picture above.  To see what it means geometrically, observe that the fiberwise restrictions $\xi_t\in \mathrm{CH}^n(\tx_t,n)$ of $\xi$ compute $\nu_{\phi}(t)$ via the composition
\begin{equation}\label{eAJ}
\mathrm{CH}^n(\tx_t,n)\overset{\mathrm{AJ}}{\to}J(H^{n-1}(\tx_t)(n))\to J(W_{n-1}H^{n-1}(X_t^*))\hookrightarrow J(\overline{H}^{n-1}(X_t^*)(n)).
\end{equation}
So the ``partial extensions'' of $\QQ(-n)$ by $\overline{H}^{n-1}(X_t^*)/W_{n-1}$ in \eqref{e43b} also split ($\forall t$). The upshot is that the extension classes of both \eqref{e43b} and \eqref{eext} reduce to (the image of) the class of \eqref{eext3}, geometrically described by $\mathrm{AJ}(\xi_t)$ in \eqref{eAJ}.
\end{proof}

Having exhibited $\cv_{\phi}^*$ as the \emph{regulator extension}, we turn to its dual
\begin{equation}\label{e43c}
0\to \QQ(0)\to(\cv^{*}_{\phi,t})^{\vee}(-n)\to\underline{H}_{n-1}(X_t^*)(-n)\to 0,	
\end{equation}
which identifies with the localization sequence
\begin{equation*}
0\to H^n(\PP_{\Delta},\DD_{\Delta})\to H^n(\PP_{\Delta}\setminus X_t,\DD_{\Delta}\setminus Z) \overset{\text{Res}}{\to}\ker\{H^{n-1}(X_t,Z)\overset{\imath_*}{\to} H^{n+1}(\PP_{\Delta},\DD_{\Delta})\}(-1)\to 0.
\end{equation*}
Writing $\Omega_t :=\frac{\mathrm{dlog}(\ux)}{1-t\phi(\ux)}\in \Omega^n(\PP_{\Delta}\setminus X_t)$ (so that $(2\pi\ay)^{n-1}\omega_t=\mathrm{Res}(\Omega_t)$), we obtain periods of the extension by lifting $(2\pi\ay)^{n-1}[\omega_t]\in F^n\{\ker(\imath_*)(-1)\}$ to $[\Omega_t]\in F^n H^n(\PP_{\Delta}\setminus X_t,\DD_{\Delta}\setminus Z;\CC)$ and pairing with the lift of $1^{\vee}\in\QQ(0)^{\vee}$ to $T_{\ux}:=\cap_{i=1}^n T_{x_i}\in H_n(\PP_{\Delta}\setminus X_t,\DD_{\Delta}\setminus Z;\QQ)$.  This yields
\begin{equation}\label{e43d}
\int_{({-}1)^{n{-}1} T_{\ux}}\Omega_t = \int_{\PP_{\Delta}}\tfrac{d[R_n]}{({-}2\pi\ay)^n}\wedge \Omega_t =\int_{\PP_{\Delta}}{R_n}\wedge \tfrac{d[\Omega_t]}{(2\pi\ay)^n}\equiv{\langle\tilde{\nu}_{\phi}(t),[\omega_t]\rangle},
\end{equation}
where the last equality only holds if $T_{\ux}\cap X_t^*=\emptyset$, and only modulo \emph{relative} periods $(2\pi\ay)^n\int_{\eta}\omega_t$ (with $\eta\in H_{n-1}(X_t,Z;\QQ)$).

\begin{rem}\label{r43a}
The left-hand term of \eqref{e43d} is a special case of the GKZ integral \eqref{e22f}, but with non-general $\phi$.  This type of integral also appears in Feynman integral computations \cite{BKV1,BKV2}.
\end{rem}

When $\phi$ is tempered (so that $\tilde{\nu}_{\phi}(t)\in H^{n-1}(\tx_t,\CC)$), and certain technical assumptions hold (cf. \cite[$\S$4.2]{BKV1}), the last equality of \eqref{e43d} holds modulo usual periods of $\omega_t$.   The VMHS picture, for $(\cv_{\phi}^*)^{\vee}(-n)$ as well as $\cv_{\phi}^{\vee}(-n)$, is of course dual to that above:\footnote{In view of the proof of Lemma \ref{lnew} and the polarization, we have $\gr^W_{n-1}H^{n-1}(X_t,Z)\cong\ch_{v,t}^{n-1}$.}
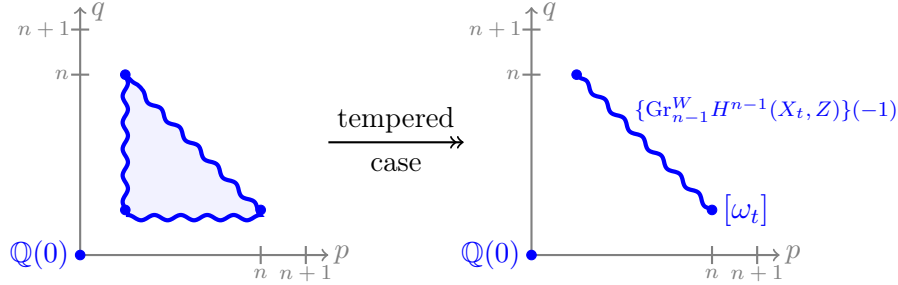
\[\begin{tikzpicture}
[scale=0.6]
\draw [<->,thick,gray] (0,5.5) -- (0,0) -- (5.5,0);
\draw [<->,thick,gray] (10,5.5) -- (10,0) -- (15.5,0);
\draw [->>,thick,black] (5.5,2.5) -- (8.5,2.5);
\draw [thick,gray] (4,-0.2) -- (4,0.2);
\draw [thick,gray] (5,-0.2) -- (5,0.2);
\draw [thick,gray] (-0.2,5) -- (0.2,5);
\draw [thick,gray] (-0.2,4) -- (0.2,4);
\node [gray] at (5,-0.4) {\tiny $n+1$};
\node [gray] at (4,-0.4) {\tiny $n$};
\node [gray] at (-0.8,5) {\tiny $n+1$};
\node [gray] at (-0.4,4) {\tiny $n$};
\node [gray] at (5.8,0) {\small $p$};
\node [gray] at (0.4,5.4) {\small $q$};
\node [] at (7,3) {\small tempered};
\node [] at (7,2) {\small case};
\filldraw [color=blue, fill=blue!5, ultra thick, decorate, decoration={snake,amplitude=1pt}] (1,4) -- (1,1) -- (4,1) -- (1,4);
\filldraw [blue] (0,0) circle (3pt) node [anchor=east] {\small $\QQ(0)$};
\filldraw [blue] (1,1) circle (3pt);
\filldraw [blue] (4,1) circle (3pt);
\filldraw [blue] (1,4) circle (3pt);
\draw [thick,gray] (14,-0.2) -- (14,0.2);
\draw [thick,gray] (15,-0.2) -- (15,0.2);
\draw [thick,gray] (9.8,5) -- (10.2,5);
\draw [thick,gray] (9.8,4) -- (10.2,4);
\node [gray] at (15,-0.4) {\tiny $n+1$};
\node [gray] at (14,-0.4) {\tiny $n$};
\node [gray] at (9.2,5) {\tiny $n+1$};
\node [gray] at (9.6,4) {\tiny $n$};
\node [gray] at (15.8,0) {\small $p$};
\node [gray] at (10.4,5.4) {\small $q$};
\draw [color=blue, ultra thick, decorate, decoration={snake,amplitude=1pt}] (14,1) -- (11,4);
\filldraw [blue] (10,0) circle (3pt) node [anchor=east] {\small $\QQ(0)$};
\filldraw [blue] (11,4) circle (3pt);
\filldraw [blue] (14,1) circle (3pt) node [anchor=west] {$[\omega_t]$};
\node [blue] at (15.2,3.2) {\tiny $\{\gr^W_{n-1}H^{n-1}(X_t,Z)\}(-1)$};
\end{tikzpicture}\]
It will be convenient to enshrine the right-hand term of \eqref{e43d} in a
\begin{defn}\label{d43a}
The \emph{truncated higher normal function} (THNF) associated to a tempered $\phi$ is (any branch of) the multivalued function $V_{\phi}(t):=\langle\tilde{\nu}^v_{\phi}(t),[\omega_t]\rangle$.
\end{defn}

\noindent Later we shall choose a branch of $V_{\phi}$; but independent of this choice, it follows from \cite{dAMS} that the THNF satisfies an inhomogeneous Picard-Fuchs equation
\begin{equation}\label{e43e}
LV_{\phi}(t)=g_{\phi}(t)	
\end{equation}
where $g_{\phi}\in \bar{\QQ}(t)$ and $L$ (from $\S$\ref{SIIB}) depend only on $\phi$.

\begin{rem}\label{r43b}
Suppose $\mathrm{rk}(\HH_v^{n-1})=n$, write $L=\sum_{i=0}^n q_{n-i}(t)\delta^i_t$ (with $q_0(0)=1$), and let $\mathcal{Y}(t):=\langle (2\pi\ay)^{n-1}\omega_t,\nabla_{\delta_t}^{n-1}\omega_t\rangle$ denote the Yukawa coupling.  Taking $\gamma_t^{\vee}\in (\HH_v^{n-1})_{T_0}$ a local generator with $\langle\gamma_t,\gamma_t^{\vee}\rangle =1$, define $D_{\phi}\in\QQ^*$ by $N_0^{n-1}\gamma_t^{\vee}=:D_{\phi}\gamma_t$.  By \cite[Cor. 4.5]{DK1}, we have $g_{\phi}(t)=q_0(t)\mathcal{Y}(t)$.  Moreover, if the $\{\tilde{X}_{\sigma}\}_{\sigma\in \Sigma\setminus\{0,\infty\}}$ have only nodal singularities, then by \cite[Prop. 7.1]{Ke20} $\mathcal{Y}(t)=\tfrac{D_{\phi}}{q_0(t)}$.
\end{rem}

\subsection{$\cv_{\phi}$ at infinity}\label{SIIID}

While $\nu_{\phi}$ is singular at $0$, we can compute its limit at $t=\infty$.  First we shall isolate a part of the extension that splits off whether or not $\phi$ is tempered (which we don't assume here).
\begin{defn}\label{d44a}
For $\sigma\in\Sigma$, the (pure weight $\ell$) \emph{phantom cohomology}
\begin{equation}\label{e44a}\small
\mathrm{Ph}^{\ell}_{\sigma}:=\ker\{H^{\ell}(X_{\sigma})\to\psi_{\sigma}\ch^{\ell}\}=\text{im}\{H_{2n-\ell}(X_{\sigma}){\small(-n)}\to H^{\ell}(X_{\sigma})\}
\end{equation}\normalsize
measures the cycles that vanish on the nearby fiber.  For any subset $\Sigma'\subseteq\Sigma$ (e.g. $\Sigma^*:=\Sigma\setminus\{0,\infty\}$), put $\mathrm{Ph}^{\ell}_{\Sigma'}:=\oplus_{\sigma\in\Sigma'}\mathrm{Ph}^{\ell}_{\sigma}$.
\end{defn}
\noindent (We shall also write $\cx_{S}$ resp.~$\tx_{S}$ for $\pi^{-1}(S)$ when $S$ is open resp.~finite.)
\begin{prop}\label{p44a}
In $\mathrm{AVMHS}(\cu)$ we have $\cv_{\phi}=\mathcal{A}^{\dagger}_{\phi}\oplus \mathrm{Ph}^n_{\Sigma\setminus\{0\}}$, where $\mathrm{Ph}^n_{\Sigma\setminus\{0\}}$ is constant of weight $n$, and $\mathcal{A}^{\dagger}_{\phi,t}$ is an extension of $\mathrm{IH}^1(\PP^1\setminus\{0\},\ch^{n-1}_v)$ \textup{(}also constant, but mixed\textup{)} by $\ch_v^{n-1}$.	Viewing $\mathcal{A}_{\phi}^{\dagger}$ instead as an extension of $\QQ(-n)$ recovers $\nu_{\phi}$.
\end{prop}
\begin{proof}
By the Decomposition Theorem (cf. \cite[(5.9)]{KL}), for any proper algebraic subset $\cs\subset\PP^1$ we have
\begin{equation}\label{e44b}
H^{\ell}(\cx_{\cs})\cong H^{\ell}_f \oplus \mathrm{IH}^1(\cs,\ch^{\ell-1})\oplus \mathrm{Ph}^{\ell}_{\Sigma\cap\cs}	
\end{equation}
as MHS.  (See [op.~cit., Prop.~5.5(i)] for the fact that $\mathrm{Ph}^{\ell}_{\Sigma\cap \cs}$ is pure of weight $\ell$, and also of level $\leq \ell-2$.) The long exact sequence associated to $(\cx_{\cs},X_t)$ (for $t\in\cu$) therefore exhibits $H^n(\cx_{\cs},\tx_t)$ as an extension of $\mathrm{IH}^1(\cs,\ch^{n-1})\oplus \mathrm{Ph}^n_{\Sigma\cap\cs}$ by $\ch^{n-1}_v$.  But as a sub-MHS of $H^n(\cx_{\cs})$, $\mathrm{Ph}_{\Sigma\cap\cs}^n$ is the image of $H_n(\tx_{\Sigma\cap\cs})(-n)\cong H^n_{\tx_{\Sigma\cap\cs}}(\cx_{\cs})\cong H^n_{\tx_{\Sigma\cap\cs}}(\cx_{\cs},\tx_t)$ under the Gysin map, which obviously factors through $H^n(\cx_{\cs},X_t)$, splitting that part of the extension. Specializing to $\cs=\PP^1\setminus\{0\}$, we have $\mathrm{IH}^1(\cs,\ch^{n-1}_f)\cong H^1(\cs)\otimes H^{n-1}_f=\{0\}$, and so $\mathrm{IH}^1(\cs,\ch^{n-1})=\mathrm{IH}^1(\cs,\ch^{n-1}_v)$.  Finally, since $\mathcal{A}_{\phi}^{\dagger}$ already gives the class of $\cv_{\phi}$ in LHS\eqref{bigT}, applying $\Theta_*$ yields $\nu_{\phi}$.
\end{proof}

\begin{rem}\label{r44a}
An initial observation about $\mathcal{A}_{\phi}^{\dagger}=\cv_{\phi}/\mathrm{Ph}^n_{\Sigma\setminus\{0\}}$ is that it elucidates what temperedness achieves.  Comparing with the discussion around Lemma \ref{lnew}, we see that it is built out of the three parts
\begin{align}
\tfrac{W_{2n}}{W_{2n-2}}\mathcal{A}_{\phi}^{\dagger}&\cong \gr^W_{2n}\mathrm{IH}^1(\PP^1\setminus\{0\},\ch_v^{n-1})\cong\QQ(-n) \label{e44'} \\
\tfrac{W_{2n-2}}{W_{n-1}}\mathcal{A}_{\phi}^{\dagger}&\cong W_{2n-2}\mathrm{IH}^1(\PP^1\setminus\{0\},\ch_v^{n-1})\;\left(\;\cong \overline{H}^{n-1}(X_t^*)/W_{n-1} \text{ if $\phi$ is $\Delta$-regular}\right)\label{e44''}\\
\gr^W_{n-1}\mathcal{A}_{\phi}^{\dagger}&\cong \ch^{n-1}_v.\label{e44'''}
\end{align}
Temperedness splits the extension of \eqref{e44'} by \eqref{e44''}, which is a \emph{constant} extension since it appears inside $\mathrm{IH}^1(\PP^1\setminus\{0\},\ch^{n-1}_v)$.
\end{rem}

Note that if $\sigma\in\Sigma\cap\cs$, the same computation (together with\footnote{Here we need not assume unipotent monodromies; see \cite{KL}.} Clemens-Schmid) exhibits $H^n(\cx_{\cs},\tx_{\sigma})$ as the direct sum of $\mathrm{Ph}^n_{\Sigma\cap\cs\setminus\{\sigma\}}$ with an extension of $\mathrm{IH}^1(\cs,\ch^{n-1})$ by $(\psi_{\sigma}\ch^{n-1}_v)^{T_{\sigma}}$.  When $\cs=\PP^1\setminus\{0\}$ this yields the
\begin{cor}\label{c44a}
The \textup{MHS} $\mathsf{A}_{\phi}^{\dagger}:=H^n(\cx\setminus\tx_0,\tx_{\infty})/\mathrm{Ph}_{\Sigma^*}$	 is isomorphic to $(\psi_{\infty}\mathcal{A}^{\dagger}_{\phi})^{T_{\infty}}$, hence computes $\lim_{\infty}\nu_{\phi}$.
\end{cor}
\noindent Dually, we may define
\begin{equation}\label{e44c}
\mathsf{A}_{\phi}:=H^n(\cx\setminus \tx_{\infty},\tx_0)/\mathrm{Ph}_{\Sigma^*} \cong(\mathsf{A}_{\phi}^{\dagger})^{\vee}(-n),
\end{equation}
which is itself obtained as the limit at $0$ (more precisely, $(\psi_0\mathcal{A}_{\phi})^{T_0}$) of
\begin{equation}\label{e44d}
\mathcal{A}_{\phi,t}:=H^n(\cx\setminus \tx_{\infty},\tx_t)	/\mathrm{Ph}_{\Sigma\setminus\{\infty\}}.
\end{equation}
\begin{example}\label{ex44a}
For the six $n=2$ Laurent polynomials in Example \ref{ex31a}, the (weak) Fano varieties\footnote{This is meant in the sense of Conjectures \ref{c32a} and \ref{c41a} only, although Conjecture \ref{c33a} \textit{bis} does in fact hold (modulo constant terms) for $\phi^{(i)}$ if $i=1,3,4,6$. (The nontempered examples were included for variety.) Note that the cases $i=4,5$ are weak Fano (not Fano) since they are resolutions of singular toric surfaces.} and VMHS Hodge-Deligne diagrams are:
\[\begin{tikzpicture}
\draw [gray] (-0.2,-0.3) -- (12.4,-0.3);
\draw [gray] (0.2,0.1) -- (0.2,-15);
\draw [gray] (1.7,0.1) -- (1.7,-15);
\draw [gray] (2.85,0.1) -- (2.85,-15);
\node [] at (0,0) {$i$};
\node [] at (1,-0.05) {\tiny tempered?};
\node [] at (1,-1.5) {Y};
\node [] at (2.3,-1.5) {$\PP^2$};
\node [] at (1,-4) {N};
\node [] at (2.3,-4) {$\mathbb{F}_1$};
\node [] at (1,-6.5) {Y};
\node [] at (2.3,-6.5) {$\mathrm{dP}_5$};
\node [] at (1,-9) {Y};
\node [] at (2.3,-9.1) {$\PP_{{}_{[1:2:3]}}$};
\node [] at (1,-11.5) {N};
\node [] at (2.3,-11.5) {$(\PP^2)^{\circ}$};
\node [] at (1,-14) {Y};
\node [align=center] at (2.2,-14) {$\mathrm{dP}_3$};
\node [] at (2.3,0) {$\mathrm{X}^{\circ}$};
\node [] at (0,-1.5) {$1$};
\node [] at (0,-4) {$2$};
\node [] at (0,-6.5) {$3$};
\node [] at (0,-9) {$4$};
\node [] at (0,-11.5) {$5$};
\node [] at (0,-14) {$6$};
\node [] at (4.2,0) {$\cv_{\phi}$};
\node [] at (5.4,0) {$\supseteq$};
\node [] at (6.6,0) {$\mathcal{A}_{\phi}^{\dagger}$};
\node [] at (7.8,0) {$\overset{\psi_{\infty}^{T_{\infty}}}{\longrightarrow}$};
\node [] at (9,0) {$\mathsf{A}_{\phi}^{\dagger}$};
\node [] at (10.2,0) {\small $\overset{(\psi_0^{T_0})^{\vee}(-n)}{\longleftarrow}$};
\node [] at (11.4,0) {$\mathcal{A}_{\phi}$};
\foreach \i in {0,...,5} {\foreach \j in {0,...,3} {\draw [thick,gray] (3.2+\j*2.4,-0.5-\i*2.5) -- (3.2+\j*2.4,-2.5-\i*2.5) -- (5.2+\j*2.4,-2.5-\i*2.5);}}
\foreach \j in {0,1,2} {\foreach \i in {0,...,5} {\filldraw [blue] (5+\j*2.4,-0.7-2.5*\i) circle (2pt);}}
\foreach \j in {0,1,3} {\foreach \i in {0,...,5} {\filldraw [blue] (3.2+\j*2.4,-1.6-2.5*\i) circle (2pt); \filldraw [blue] (4.1+\j*2.4,-2.5-2.5*\i) circle (2pt);}}
\foreach \i in {1,...,5} {\filldraw [blue] (4.1,-1.6-2.5*\i) circle (2pt);}
\node [blue] at (4.3,-6.6) {$4$};
\node [blue] at (4.3,-9.1) {$3$};
\node [blue] at (4.3,-11.6) {$6$};
\node [blue] at (4.3,-14.1) {$6$};
\filldraw [blue] (8,-1.6) circle (2pt);
\filldraw [blue] (8.9,-2.5) circle (2pt);
\filldraw [blue] (12.2,-1.6) circle (2pt);
\filldraw [blue] (11.3,-0.7) circle (2pt);
\filldraw [blue] (8.9,-4.1) circle (2pt);
\filldraw [blue] (11.3,-4.1) circle (2pt);
\filldraw [blue] (6.5,-4.1) circle (2pt);
\filldraw [blue] (12.2,-5.7) circle (2pt);
\filldraw [blue] (8,-7.5) circle (2pt);
\foreach \j in {1,2,3} {\filldraw [blue] (4.1+\j*2.4,-9.1) circle (2pt); \node [blue] at (4.3+\j*2.4,-9.1) {$3$};}
\filldraw [blue] (12.2,-8.2) circle (2pt);
\filldraw [blue] (8,-10) circle (2pt);
\foreach \j in {1,2,3} {\filldraw [blue] (4.1+\j*2.4,-11.6) circle (2pt); \node [blue] at (4.3+\j*2.4,-11.6) {$6$};}
\filldraw [blue] (12.2,-11.6) circle (2pt);
\filldraw [blue] (11.3,-10.7) circle (2pt);
\filldraw [blue] (8.9,-12.5) circle (2pt);
\filldraw [blue] (8,-11.6) circle (2pt);
\foreach \j in {0,1,2} {\foreach \i in {1,4} {\draw [red,<->] (4.3+\j*2.4,-1.4-\i*2.5) -- (4.8+\j*2.4,-0.9-\i*2.5);}}
\end{tikzpicture}\]
The red arrows in cases (2) and (5) denote nontorsion extensions of \eqref{e44'} by \eqref{e44''}, reflecting the nontemperedness.  (These extensions record, in $\mathrm{Ext}_{\text{MHS}}^1(\QQ(-2),\QQ(-1))\cong \CC/\QQ(-1)$, the logarithms of the toric boundary coordinates of the base locus $X_t\cap \DD_{\Delta}$.) In the other cases, the limit $\mathsf{A}^{\dagger}_{\phi}$ contains only torsion extensions.\footnote{In case (1) ($X_{\infty}$ smooth) this is by a computation in $K_2(X_{\infty})$; in (3) and (4) ($X_{\infty}$ singular), it is because $K_3^{\text{ind}}(\QQ)$ is torsion.  Later we will see how torsion extensions actually may lift to well-defined invariants in $\CC$.}  In \emph{all} cases, we have $\mathrm{rk}(\gr_F^k \cv_{\phi})=\dim (H^{k,k}(\xc))$ in accordance with Conjecture \ref{c41a}.  

More striking is the disparity in form between $\mathcal{A}_{\phi}^{\dagger}$ and $\mathcal{A}_{\phi}$.  While both share $\gr^W_1\cong \ch^1$ and $\gr^W_2\cong \IH^1(\PP^1,\ch^1)$, we have $\tfrac{W_4}{W_2}\mathcal{A}^{\dagger}_{\phi}\cong \QQ(-2)$ vs. $\tfrac{W_4}{W_2}\mathcal{A}_{\phi}\cong (\psi_{\infty}\ch^1)_{T_{\infty}}(-1)$.  Only in cases (3) and (4) does $\mathcal{A}_{\phi}$ yield a ``$K_2$-type'' normal function in $\mathrm{ANF}(\ch^1(2))$, which for (3) is due to an involution of $\cx$ over $t\mapsto -\tfrac{1}{t}$ \cite[$\S$5.2]{Ke17}.  However, we may regard (in (2), (4), and (5)) extensions of $\QQ(-1)\subseteq \IH^1(\PP^1,\ch^1)$ by $\ch^1$ as ``$K_0$-type'' normal functions, whose image in $\mathrm{ANF}(\ch^1(1))$ generate the Mordell-Weil group ($\otimes \QQ$) of $\pi$.  That their limits at $0$ capture part (or all, as in (2)) of $\lim_{\infty}\nu_{\phi}$ is essentially the fact that the limits in $X_0$ of Abel-Jacobi of differences of sections in $X_t\cap\DD_{\Delta}$ are given by ratios of toric coordinates on $\DD_{\Delta}$.
\end{example}
As mentioned in the Introduction, $\mathcal{A}_{\phi}^{\dagger}$ and $\mathcal{A}_{\phi}$ do not share the dual relationship \eqref{e44c} with their limits.  Indeed, as we have just seen, $\mathcal{A}_{\phi}$ is not even an HNF in a canonical way (unlike $\mathcal{A}_{\phi}^{\dagger}$).  However, for $n\geq 3$, it is precisely this lack of canonicity which makes $\mathcal{A}_{\phi}$ better adapted to exhibiting $\lim_{\infty}\nu_{\phi}$ in terms of limits of truncated HNFs.

\section{The Problem and some Fano threefold examples}\label{SIV}

In this section we state a precise but restricted version of the Arithmetic Mirror Symmetry Problem (see \S\ref{SIV2}), and then solve it when $\mathrm{X}^{\circ}$ is one of the Mukai Fano threefolds $V_{2N}$ ($5\leq N\leq 9$) \cite{Gv}.  For each of these, \cite{Fano} provides many LG-models of the form in Definition \ref{d31a} -- corresponding to the many possible toric degenerations $\mathbb{P}_{\Delta^{\circ}}$ of $\mathrm{X}^{\circ}$ -- satisfying Conjectures \ref{c32a} and \ref{c33a} \textit{bis}.  They are found by taking $\phi$ to be (up to an additive constant) the Minkowski polynomial \cite{CCGGK} for the corresponding (reflexive) $\Delta$, which is tempered in view of \cite[Prop. 2.4]{dS2}.

Our job is then to exhibit $\mathcal{A}_{\phi,t}$ as a geometric HNF in the sense of \eqref{e42c}, and the Ap\'ery constant $\alpha_{\mathrm{X}^{\circ}}$ as the limit at $t=0$ of the corresponding THNF, canonically normalized as described in \S\ref{SIV1}.  In contrast to $\mathcal{A}^{\dagger}_{\phi,t}$, this cannot arise from the lift of the coordinate symbol $\{x_1,x_2,x_3\}$ to $\mathrm{CH}^3(\tilde{X}_t,3)$, since that HNF is singular at $0$.  Rather, we are looking for an extension 
\begin{equation}\label{e50a}
0\to \ch^2_v(p)\to\mathcal{A}_{\phi,t}\to \QQ(0)\to 0
\end{equation}
arising from $$\mathcal{Z}\mapsto\nu_{\mathcal{Z}}\colon \mathrm{CH}^p(\cx\setminus \tilde{X}_{\infty},r-1)\to \mathrm{ANF}(\ch^{2p-r}_v(p))$$ with $2p-r=2$, which forces $(p,r)=(3,3)$ ($\mathcal{Z}_t$ belongs to the $K^{\text{alg}}_3$ of the $K3$ fibers $\tilde{X}_t$) or $(2,1)$ ($\mathcal{Z}_t$ lies in $K_1$ of the fibers).\footnote{Taking $p>3$ yields $F^{-1}\ch^2_v(p)=\{0\}$, making the extension class of \eqref{e50a} horizontal (by transversality) with rational monodromy (images under $T_{\sigma}-I$), hence trivial (since monodromy acts irreducibly on $\ch^2_v$).} It is these cycles $\mathcal{Z}$ which (in \S\S\ref{SIVA}-\ref{SIVC}) we will show how to construct in each case.

\subsection{The inhomogeneous equation of a normal function}\label{SIV1}

Given $\nu\in \mathrm{ANF}(\ch^{n-1}_v(p))$, let $\tilde{\nu}:=\nu_{\QQ}-\nu_F$ be a multivalued holomorphic lift to $\ch^{n-1}_v$.  (Here $v$ can be a higher or classical normal function, i.e. $p\geq \tfrac{n}{2}$.) We may generalize Definition \ref{d43a} and \eqref{e43e} by setting $$V(t):=\langle \tilde{\nu}(t),[\omega_t]\rangle$$ and $g(t):=LV(t)\in\CC(t)$, which is zero iff $\nu$ is torsion \cite{dAMS}. 
(Note that since $\langle F^1,\omega\rangle=0$ and $L\langle \HH^{n-1}_v,\omega\rangle=0$, $g$ is independent of the choices of $\nu_{\QQ}$ and $\nu_F$.)  In a special case, in which $\nu$ is singular at $0$, we have a formula for $g(t)$ (Remark \ref{r43b}).

The next result summarizes what we can say more generally about this inhomogeneous term.  It is motivated as follows.  Suppose $\nu$ is nonsingular at $0$ (Definition \ref{d42b}), so that the truncated NF has a power-series expanion $V(t)=\sum_{k\geq 0}v_k t^k$ there.   If one knows $L$ and can bound the degree of $g$ (by some $m$), then we only need $\{v_k\}_{k=0}^m$ to compute $g$.

For its statement, we shall assume only that:
\begin{itemize}[leftmargin=0.5cm]
\item $\{\tilde{X}_t\}$ is a family of CY $(n{-}1)$-folds over $\PP^1$ (smooth off $\Sigma$);
\item $\{\omega_t\}$ is a section of $\ch^{n-1,0}_{v,e}\cong \co_{\PP^1}(h)$, with divisor $h[\infty]$;
\item $L=\sum_{j=0}^d t^j P_j(\delta_t)\in \CC[t,\delta_t]$ is its PF operator, of degree $d$; and
\item $\HH_v^{n-1}$ has maximal unipotent monodromy at $0$.
\end{itemize}
This is somewhat more general than the setting of the rest of this paper, which takes $\{\tilde{X}_t\}$ to arise from the level sets of a Laurent polynomial; in this case we have $h=1$ (see \cite[Ex. 4.5]{Ke20}), and frequently only nodal singularities on the $\{\tilde{X}_{\sigma}\}_{\sigma \in \Sigma^*}$.

\begin{thm}\label{t51}
Assume $\nu$ is nonsingular away from $0$ and $\infty$.  Then $g$ is a polynomial of degree $\leq d-h$.  If $\nu$ is also nonsingular at $0$, then $t\mid g$.  If $\nu$ is also nonsingular at $\infty$ and $T_{\infty}$ is unipotent, then $\deg(g)\leq d-h-1$.
\end{thm}

\begin{proof}
Let $u$ be a local coordinate on a disk $D_{\sigma}$ about $\sigma\in\Sigma$, and $\ch_e$ resp. $\ch^e$ the canonical resp. dual-canonical extensions of $\ch^{n-1}_v |_{D_{\sigma}^*}$ to $D_{\sigma}$.  (That is, the eigenvalues of $\nabla_{\delta_u}$ are in $(-1,0]$ resp. $[0,1)$.)  Assuming $\nu$ is nonsingular at $\sigma$, we may choose $\nu_{\QQ}$ so that $N_{\sigma}\nu_{\QQ}=0$; thus $\tilde{\nu}$ is $T_{\sigma}$-invariant, and extends to a section of $\ch^e$.  Since $\omega$ is a section of $\ch_e$, and $\langle \,,\,\rangle$ extends to $\ch^e\times\ch_e\to \co$, $V=\langle \tilde{\nu},\omega\rangle $ extends to a holomorphic function on $D_{\sigma}$.  For $\sigma\in \Sigma^*$, we have $L\in \CC[u,\partial_u]$ hence $g|_{D_{\sigma}}$ holomorphic.  At $\sigma=0$, maximal unipotency forces the indicial polynomial $P_0(T)$ to be divisible by $T$, so that $L$ sends $\co(D_0)\to t\co(D_0)$ and $g(0)=0$.  If $\sigma=\infty$ and $u=t^{-1}$, our assumption that $(\omega)=h[\infty]$ gives $V|_{D_{\infty}}\in u^h \co(D_{\infty})$; applying $L=\sum_{j=0}^d u^{-j} P_j(-\delta_u)$ yields $g|_{D_{\infty}}\in u^{h-d}\co(D_{\infty})$.

We can refine the result at $\infty$, and deal with singularities at $0$ and $\infty$, by writing $\tilde{\nu}$ and $\omega$ locally in terms of bases of the canonical extension.  With $\sigma,u$ as above, $\HH_{v,\CC}^{n-1}=\HH_{\CC}^{\text{un}}\oplus \HH_{\CC}^{\text{non}}$ decomposes into unipotent ($T_{\sigma}^{\text{ss}}$-invariant) and nonunipotent parts, with (multivalued) bases $\{\mathsf{e}_i\}$ and $\{\mathsf{e}_j^*\}$, the latter chosen so that $T^{\text{ss}}_{\sigma}\mathsf{e}_j^*=\zeta_k^{a_j}\mathsf{e}_j^*$ ($\zeta_k:=e^{\frac{2\pi\ay}{k}}$).  Writing $\ell(u):=\tfrac{\log(u)}{2\pi\ay}$, a basis of $\ch_e=\ch_e^{\text{un}}\oplus \ch_e^{\text{non}}$ is given by $\tilde{\mathsf{e}}_i:=e^{-\ell(u)N_{\sigma}}\mathsf{e}_i$ and $\tilde{\mathsf{e}}_j^*:=e^{-\ell(u)N_{\sigma}}u^{-\frac{a_i}{k}}\mathsf{e}_j^*$, which have the property that $\nabla_{\delta_u}\tilde{\mathsf{e}}_i\,,\nabla_{\delta_u}\tilde{\mathsf{e}}_j^*\in \ch_e$.  Admissibility says that the Hodge lift takes the form $$\nu_F(u)=u\sum_i f_i(u)\tilde{\mathsf{e}}_i + u\sum_j f_j^*(u)\tilde{\mathsf{e}}_j^* + \tilde{\mathsf{e}}_{\CC}\in \Gamma (D_{\sigma},\mathcal{V}^e),$$ where $\mathsf{e}_{\CC}$ is a $\CC$-lift of $1$ to $\VV_{\CC}^{\text{un}}$ and $\tilde{\mathsf{e}}_{\CC}:=e^{-\ell(u)N_{\sigma}}\mathsf{e}_{\CC}$.  If $\nu$ is nonsingular at $\sigma$, then $\tilde{\mathsf{e}}_{\CC}=\mathsf{e}_{\CC}$ and $\nabla_{\delta_u}\tilde{\mathsf{e}}_{\CC}=0$; if it is singular at $\sigma$, then we may assume $\tilde{\mathsf{e}}_{\CC}=\mathsf{e}_{\CC}+\ell(u)\tilde{\mathsf{e}}_1$, so that $\nabla_{\delta_u}\tilde{\mathsf{e}}_{\CC}\in \ch_e^{\text{un}}$.  Write $\mathrm{ord}_{\sigma}(\omega)=:o$ (this is $h$ if $\sigma=\infty$ and $0$ if $\sigma=0$).

Replacing $\tilde{\nu}$ by $\hat{\nu}:=\mathsf{e}_{\CC}-\nu_F$ changes it by a $\CC$-period hence does not affect $g$.  Writing $L=\sum_{k\geq 0}q_k^{\sigma}(u)\delta_u^k$, we have $$g=L\langle \hat{\nu},\omega\rangle=\sum_{k\geq 1}\sum_{j=1}^k q_k^{\sigma}(u)\langle \nabla_{\delta_u}^{j}\hat{\nu},\nabla_{\delta_u}^{k-j}\omega \rangle$$ since $\nabla_L\omega=0$.  Clearly $\nabla^{k-j}_{\delta_u}\omega\in u^o \ch_e$, while $\nabla^{j}_{\delta_u}\hat{\nu}\in u\ch_e=u\ch^e_{\text{un}}\oplus \ch^e_{\text{non}}$ resp. $\ch^{\text{un}}_e\oplus \ch_e^{\text{non}}=\ch^e$ for $\nu$ nonsingular resp. singular at $\sigma$.  Hence $\langle \nabla_{\delta_u}^{j}\hat{\nu},\nabla_{\delta_u}^{k-j}\omega\rangle$ belongs to $u^{o+1}\co(D_{\sigma})$ if $\nu$ is nonsingular \emph{and} $T_{\sigma}$ is unipotent, and otherwise to $u^o\co(D_{\sigma})$.  For $\sigma=\infty$, multiplying by $q_k^{\sigma}(u)$ introduces $u^{-d}$.  The result follows.
\end{proof}

\begin{rem}\label{r51a}
Different choices of $\nu_{\QQ}$ yield branches of $V$ that differ by $\QQ(p)$-periods $(2\pi\ay)^p\int_{\varphi_t}\omega_t$, $\varphi_t\in H_{n-1}(\tilde{X}_t,\QQ)$.  If the $\{T_{\sigma}-I\}_{\sigma \in \Sigma^*}$ have rank one, and there are $d$ of them (i.e. $\ch_v^{n-1}$ has no ``removable singularities''), and one $\sigma_0\in \Sigma^*$ has greater modulus than the others, then we say that $\cx$ is of \emph{normal conifold type}.  In this case $V$ \textbf{can be chosen uniquely} by maximizing its radius of convergence; that is, there is a unique branch which is single-valued on the complement of the interval $[\sigma_0,\infty]$.
\end{rem}

\subsection{The arithmetic mirror symmetry problem}\label{SIV2}

Rather than reiterating the general but vague version from the Introduction, we give a more precise variant in a restricted setting.  Assume that our Fano variety and LG-model satisfy the following:

\begin{itemize}[leftmargin=0.5cm]
\item $H^*_{\text{prim}}(\mathrm{X}^{\circ})$ and $(\psi_0\ch^{n-1}_v)^{T_0}$ are Hodge-Tate of rank $r_0$, with isomorphic associated gradeds (as predicted by Conjecture \ref{c41a});
\item $H^n_{\text{prim}}(\mathrm{X}^{\circ})=\{0\}$ (if $n$ is even), and $\rho(\mathrm{X}^{\circ})=1$;
\item $\cx$ is of normal conifold type (Remark \ref{r51a}), and satisfies Conjecture \ref{c33a} \textit{bis};
\item $\phi$ (and thus $\cx$, and $L$) is defined over $\bar{\QQ}$;
\item $d=r_0+1$; and
\item $P_0(d-1)\neq 0$.
\end{itemize}

\noindent Referring to \S\ref{SIIC}, we write $b_j:=u^{(d-1)}_j$ and $B(t):=\sum_{j\geq d-1} b_j t^j$, so that $\alpha_{\mathrm{X}^{\circ}}^{(d-1)}=\lim_{j\to \infty}\tfrac{b_j}{a_j}$ and $LB=P_0(d-1)t^{d-1}$.

\begin{prob}\label{c52}
Exhibit the Ap\'ery constants $\{\alpha_{\mathrm{X}^{\circ}}^{(i)}\}_{i=1}^{d-1}$ as periods, by showing that:

\textup{(a)} The first $d-2$ constants $\{\alpha_{\mathrm{X}^{\circ}}^{(i)}\}_{i=1}^{d-2}$ are \textup{(}up to $\bar{\QQ}^*$-multiples\textup{)} extension classes in $(\psi_0\ch_v^{n-1})^{T_0}$ which are torsion \textup{(}i.e. powers of $2\pi\ay$\textup{)} if $\phi$ is tempered.

\textup{(b)} There is \textup{(}up to scale\textup{)} a unique HNF $\nu\in\mathrm{ANF}(\ch^{n-1}_v(p))\setminus\{0\}$ singular only at $t=\infty$, for some \textup{(}unique\textup{)} $p\in [\tfrac{n+1}{2},n]\cap \ZZ$.  This HNF is motivic, i.e. arises from some $\mathcal{Z}\in\mathrm{CH}^p(\cx\setminus \tilde{X}_{\infty},2p-n)_{\QQ}$\textup{;} and $LV=-\mathfrak{k}t^{d-1}$ for some\footnote{More precisely, $\mathfrak{k}$ should belong to the common field of definition of $\cx$ and $\mathcal{Z}$.} $\mathfrak{k}\in \bar{\QQ}^*$.

\textup{(c)} Normalize $V$ uniquely as in Remark \ref{r51a}, and set $\hat{V}(t):=\tfrac{P_0(d-1)}{\mathfrak{k}}V(t)$.  Then $\alpha_{\mathrm{X}^{\circ}}^{(d-1)}=\hat{V}(0)+\sum_{i=1}^{d-2}\beta_i \alpha_{\mathrm{X}^{\circ}}^{(i)}$, where $\beta_i\in \bar{\QQ}\langle \hat{V}(0),\hat{V}'(0),\ldots,\hat{V}^{(i)}(0)\rangle$.
\end{prob}

\begin{rem}\label{r52a}
Stated in this way, the thrust of the Arithmetic Mirror Symmetry Problem is somewhat obscured.  What it really proposes is that given a Fano $\mathrm{X}^{\circ}$ with $H_{\text{prim}}^*$ as above, of rank one less than the degree of its quantum differential equation, there \emph{exists} an LG-model $\cx$ (and cycle $\mathcal{Z}$) satisfying the remaining hypotheses together with the content of \ref{c52}(a)-(c).
\end{rem}

In the three subsections that follow, we solve this Problem in several cases with $n=3=r$ and $d=2$.  The situation simplifies, since $P_0(1)=1$ and there is only one Ap\'ery constant $\alpha_{\mathrm{X}^{\circ}}$; moreover, Theorem \ref{t51} guarantees that $LV=-\mathfrak{k}t$ for some $\mathfrak{k}\in\CC^*$ once we have a HNF of the type described.  So it remains to produce $\mathcal{Z}$ (hence $\nu$), and show $\mathfrak{k}\in\bar{\QQ}^*$ and $\alpha_{\mathrm{X}^{\circ}}=\hat{V}(0)$ in each case; we defer the uniqueness to \S\ref{SV}.

\begin{rem}\label{r52b}
A general result (for $n=r$ and $d=2$) encompassing these cases appeared in \cite[Thm. 10.11]{Ke20}, making essential use of Theorem \ref{t51} above.  It reduces the Arithmetic Mirror Symmetry Problem to establishing the existence of a ``good'' LG-model and checking the Beilinson-Hodge Conjecture; once the cycle is found, the equality $\alpha_{\mathrm{X}^{\circ}}=\hat{V}(0)$ is automatic.  But we shall explicitly compute $\hat{V}(0)$ below in each case anyway, both to provide explicit (and instructive) solutions to the Problem, and to check that $\mathfrak{k}\neq 0$ and the cycle indeed produces a nontrivial HNF.  Moreover, with [loc. cit.] in hand, one may regard these computations of $\hat{V}(0)$ as illustrations of a regulator calculus which is available for calculating the Ap\'ery constant when modular and other methods (such as in \cite{Gv}) are unavailable.
\end{rem}

\subsection{$K_1$ of a $K3$: the $V_{10}$ HNF}\label{SIVA}

The irrational Fano threefold $V_{10}=G(2,5)\cap \mathsf{Q}\cap \mathsf{P}_1\cap \mathsf{P}_2$ (quadric and linear sections of the Pl\"ucker embedding) has a mirror LG model with discriminant locus $\Sigma=\{0,\sigma_+,\sigma_-,\infty\}$ (where $\sigma_{\pm}=\tfrac{-11\pm 5\sqrt{5}}{4}$), given by the Laurent polynomial
\begin{equation}\label{e53a}
\phi(\ux)=\frac{(1-x_3)(1-x_1-x_3)(1-x_2-x_3)(1-x_1-x_2-x_3)}{-x_1x_2x_3}.	
\end{equation}
Namely, compactifying $\{1=t\phi(\ux)\}$ in $\PP_{\Delta}$ yields a family $\{X_t\}_{t\in\PP^1}$, whose fibers over $\PP^1\setminus\Sigma$ are singular $K3$s with one $A_3$ and six $A_1$ singularities; these are resolved (to Picard-rank $19$ $K3$s) 
\noindent\begin{minipage}{0.03\textwidth}
\end{minipage}
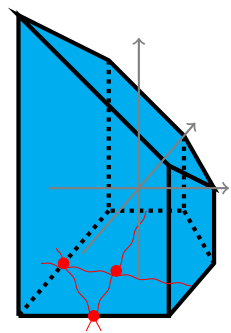
\begin{minipage}{0.20\textwidth}
\begin{tikzpicture}
\draw[black,ultra thick,fill=cyan] (-1.6,-1.7) -- (0.4,-1.7) -- (0.4,0.3) -- (-1.6,2.3) -- (-1.6,-1.7);
\draw[black,ultra thick,fill=cyan] (0.4,-1.7) -- (1,-1) -- (1,0) -- (0.4,0.3) -- (0.4,-1.7);
\draw[black,ultra thick,fill=cyan] (0.4,0.3) -- (1,0) -- (0.6,0.7) -- (-0.4,1.7) -- (-1.6,2.3) -- (0.4,0.3);
\draw[black,ultra thick,dotted] (-0.4,1.7) -- (-0.4,-0.3) -- (0.6,-0.3) -- (0.6,0.7);
\draw[black,ultra thick,dotted] (0.6,-0.3) -- (1,-1);
\draw[black,ultra thick,dotted] (-0.4,-0.3) -- (-1.6,-1.7);
\draw[red,decorate,decoration={snake,amplitude=0.5pt}] (0.1,-0.3) -- (-0.7,-1.9);
\draw[red,decorate,decoration={snake,amplitude=0.5pt}] (-1.3,-1) -- (0.7,-1.3);
\draw[red,decorate,decoration={snake,amplitude=0.5pt}] (-1.1,-0.8) -- (-0.5,-1.9);
\filldraw[red] (-0.6,-1.7) circle (2pt);
\filldraw[red] (-1,-1) circle (2pt);
\filldraw[red] (-0.3,-1.1) circle (2pt);
\draw[->, gray,thick] (-1.2,0) -- (1.2,0);
\draw[->, gray,thick] (0,-1.2) -- (0,2);
\draw[->, gray,thick] (-0.75,-0.87) -- (0.75,0.87);
\end{tikzpicture}
\end{minipage}
\begin{minipage}{0.77\textwidth}	
by $\beta\colon \tilde{X}_t\to X_t$.  The Newton polytope $\Delta$, together with a portion $$X_t \cap \{x_3=0\}=\{x_1=1\}\cup\{x_2=1\}\cup\{x_1+x_2=1\}=:\cC_1\cup \cC_2\cup \cC_3$$
of the base locus (red), are displayed in the figure. The PF operator and period sequence  are given by $$L=\delta_t^3 - 2t(2\delta_t+1)(11\delta_t^2+11\delta_t + 3) - 4t^2(\delta_t+1)(2\delta_t+3)(2\delta_t+1),$$ $$a_k:=[\phi^k]_{\uo}=\sum_{i=0}^k\sum_{j=0}^{k-i}\frac{k! 2k!}{i!^2 j!^2 (k-i)!(k-j)!(k-i-j)!}=1,\,6,\,114,\,\ldots;$$
\end{minipage} 

\noindent while the monodromy operators $T_0, T_{\pm}, T_{\infty}$ have Jordan forms $J(3),(\mathbf{-1})\oplus \mathbf{1}^2, (\mathbf{-1})^2\oplus \mathbf{1}$ and LMHS types
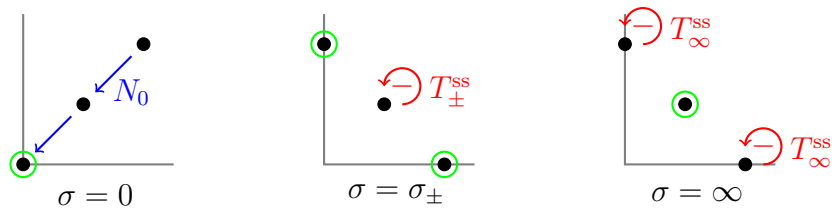
\[\begin{tikzpicture}[scale=0.8]
\node[black] at (1.2,-0.5) {$\sigma=0$};
\node[black] at (6.2,-0.5) {$\sigma=\sigma_{\pm}$};
\node[black] at (11.2,-0.5) {$\sigma=\infty$};
\draw[gray,thick] (0,2.5) -- (0,0) -- (2.5,0);
\draw[gray,thick] (5,2.5) -- (5,0) -- (7.5,0);
\draw[gray,thick] (10,2.5) -- (10,0) -- (12.5,0);
\filldraw[black] (0,0) circle (3pt);
\filldraw[black] (1,1) circle (3pt);
\filldraw[black] (2,2) circle (3pt);
\filldraw[black] (5,2) circle (3pt);
\filldraw[black] (6,1) circle (3pt);
\filldraw[black] (7,0) circle (3pt);
\filldraw[black] (10,2) circle (3pt);
\filldraw[black] (11,1) circle (3pt);
\filldraw[black] (12,0) circle (3pt);
\draw[->,thick,blue] (1.8,1.8) -- (1.2,1.2);
\draw[->,thick,blue] (0.8,0.8) -- (0.2,0.2);
\node [blue] at (1.8,1.2) {$N_0$};
\draw[green,thick] (0,0) circle (6pt);
\draw[green,thick] (5,2) circle (6pt);
\draw[green,thick] (7,0) circle (6pt);                                                         
\draw[green,thick] (11,1) circle (6pt);
\draw[->,thick,red] (6.3,1) arc (-90:180:0.3cm);
\node [red] at (6.3,1.3) {$-$};
\node [red] at (7.1,1.2) {$T^{\text{ss}}_{\pm}$};
\draw[->,thick,red] (10.3,2) arc (-90:180:0.3cm);
\node [red] at (10.3,2.3) {$-$};
\node [red] at (11.1,2.2) {$T^{\text{ss}}_{\infty}$};
\draw[->,thick,red] (12.3,0) arc (-90:180:0.3cm);
\node [red] at (12.3,0.3) {$-$};
\node [red] at (13.1,0.2) {$T^{\text{ss}}_{\infty}$};
\end{tikzpicture}\]
where the $T_{\sigma}$-invariant classes are circled.
The Ap\'ery constant is $\alpha=\tfrac{1}{10}\zeta(2)$ \cite{Gv}.

To construct the cycle $\mathcal{Z}\in \mathrm{CH}^2(\cx\setminus \tilde{X}_{\infty},1)$ we shall make use of the rational curves $\{\cC_i\}$.  On $X_t$, a higher Chow cycle is given by $(\cC_1,g_1:=\tfrac{x_2}{x_2 -1})+(\cC_2,g_2:=\tfrac{x_1 -1}{x_1})+(\cC_3,g_3:=\tfrac{x_1}{x_1-1}),$ since the sums of divisors cancel on $X_t$.  To lift this to a cycle $\mathcal{Z}_t$ on $\tilde{X}_t$ (say, for $t\notin \Sigma$), one adds two more terms $(D_1,f_1)+(D_2,f_2)$ supported on the exceptional divisors over the nodes of $X_t$ at $\cC_1\cap \cC_3$ and $\cC_2\cap \cC_3$.  These $\{\mathcal{Z}_t\}$ are the restrictions of an obvious precycle $\mathcal{Z}$ on $\cx$, whose boundary fails to vanish only on $\tilde{X}_{\infty}$.\footnote{The successive blowups along the components of the base locus occurring in the construction of $\cx$ generate additional exceptional curves on $\tilde{X}_{\infty}$ which disconnect the $5$-gon $D_1 \cup D_2 \cup \cC_1 \cup \cC_2 \cup \cC_3$, and it is on these that $\partial \mathcal{Z}$ is supported.}

The next step is to find a family of closed $2$-currents $R_t$ on $\tilde{X}_t$ representing the class $\nu_{\mathcal{Z}}(t)\in J(H^2(\tilde{X}_t)(1))$, or more precisely a lift to $H^2(\tilde{X}_t,\CC)$
which is single-valued on $D_{|\sigma_-|}$.  Writing $\mu := \{(x_1,x_2)\in \RR^2 \mid 0\leq x_2 \leq 1,\;1-x_2\leq x_1\leq 1\}$, for $|t|\ll 1$ let $\Gamma_t$ denote the branch of $\{(x_1,x_2,x_3)\in \tilde{X}_t\mid (x_1,x_2)\in \mu\}$ with $x_3$ small. Then we have\footnote{In the context of regulator currents, $\log(-)$  means the single-valued branch with discontinuity along $\RR_-$.} $R_t = (2\pi \ay)^2 \delta_{\Gamma_t}+ 2\pi\ay\sum_{i=1}^3 \log(g_i)\delta_{\cC_i} + 2\pi\ay\sum_{i=1}^2 \log(f_i)\delta_{D_i}$, which yields the THNF
\begin{align*}
V(t) = \langle [R_t],[\omega_t]\rangle = (2\pi\ay)^2 \int_{\Gamma_t} \omega_t = \tfrac{1}{2\pi\ay} \int_{\mu}\int_{|x_3|=\epsilon} \tfrac{\mathrm{dlog}(\ux)}{1-t\phi} =	\sum_{k\geq 0} t^k \int_{\mu}[\phi^k]_{x_3^0}\tfrac{dx_1}{x_1}\wedge\tfrac{dx_2}{x_2} =: \sum_{k\geq 0} v_k t^k.
\end{align*}
(Here $[-]_{x_3^0}$ takes terms of the Laurent polynomial constant in $x_3$.) By Theorem \ref{t51}, it suffices to compute
\begin{align*}
v_0 &= \int_0^1\int_{1-x_2}^1\frac{dx_1}{x_1}\wedge\frac{dx_2}{x_2}=-\int_0^1 \log(1-x_2)\frac{dx_2}{x_2}=\mathrm{Li}_2(1)=\zeta(2)\;\;\;\;\;\text{and}\\
v_1 &= \int_0^1\int_{1-x_2}^1 \left\{\begin{matrix} x_1^{-1}(4x_2^{-2}-6+2x_2)+\\(-6x_2^{-1}+6-x_2)+x_1(2x_2^{-1}-1)\end{matrix}\right\}\frac{dx_1}{x_1}\wedge\frac{dx_2}{x_2} = -10+6\zeta(2)
\end{align*}
to conclude that $LV=-10t$.  Normalization therefore yields
\begin{equation}\label{e53b}
\hat{V}(t)=\tfrac{1}{10}\zeta(2)+(-1+\tfrac{3}{5}\zeta(2))t+\cdots,
\end{equation}
as desired.
\subsection{$K_3$ of a $K3$: the $V_{12}$ HNF}\label{SIVB}

The LG mirror of the rational Fano $V_{12}=\mathit{OG}(5,10)\cap \mathsf{P}_1\cdots\cap \mathsf{P}_7$ is given by
\begin{equation}\label{e54a}
\phi(\ux)=\frac{(1-x_1)(1-x_2)(1-x_3)(1-x_1-x_2+x_1 x_2 - x_1 x_2 x_3)}{-x_1x_2x_3}.	
\end{equation}
\noindent\begin{minipage}{0.03\textwidth}
\end{minipage}
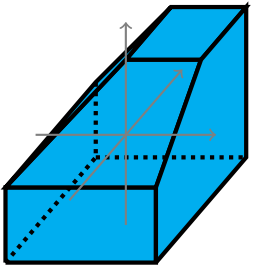
\begin{minipage}{0.3\textwidth}
\begin{tikzpicture}
\draw[black,ultra thick,fill=cyan] (-1.6,-1.7) -- (0.4,-1.7) -- (0.4,-0.7) -- (-1.6,-0.7) --(-1.6,-1.7);
\draw[black,ultra thick,fill=cyan] (0.4,-0.7) -- (1,1) -- (0,1) -- (-1.6,-0.7) -- (0.4,-0.7);
\draw[black,ultra thick,fill=cyan] (0,1) -- (1,1) -- (1.6,1.7) -- (0.6,1.7) --(0,1);
\draw[black,ultra thick,fill=cyan] (-1.6,-0.7) -- (-0.4,0.7) -- (0.6,1.7) -- (0,1) -- (-1.6,-0.7);
\draw[black,ultra thick,fill=cyan] (0.4,-1.7) -- (1.6,-0.3) -- (1.6,1.7) -- (1,1) -- (0.4,-0.7) -- (0.4,-1.7);
\draw[black,ultra thick,dotted] (-0.4,-0.3) -- (1.6,-0.3);
\draw[black,ultra thick,dotted] (-1.6,-1.7) -- (-0.4,-0.3) -- (-0.4,0.7);
\draw[->, gray,thick] (-1.2,0) -- (1.2,0);
\draw[->, gray,thick] (0,-1.2) -- (0,1.5);
\draw[->, gray,thick] (-0.75,-0.87) -- (0.75,0.87);
\end{tikzpicture}
\end{minipage}
\begin{minipage}{0.67\textwidth}	
This time the Picard-rank $19$ $K3$s $\tilde{X}_t$, smooth for $t\notin \Sigma=\{0,\sigma_+,\sigma_-,\infty\}$ ($\sigma_{\pm}=(-1\pm\sqrt{2})^4$), resolve 7 $A_1$ singularities on $X_t$.  The family $\cx$ is birational to that of \cite{BP} and underlies the proof of irrationality of $\zeta(3)$ \cite{Ke17}; indeed, $\alpha=\tfrac{1}{6}\zeta(3)$ \cite{Gv}.  Its PF operator
$$L=\delta_t^3 - t(2\delta_t+1)(17\delta_t^2+17\delta_t + 5) + t^2(\delta_t+1)^3$$
\end{minipage}

\noindent \begin{minipage}{0.7\textwidth}and (Ap\'ery) period sequence $$a_k:=\sum_{\ell=0}^k {k\choose \ell}^2 {k+\ell \choose \ell}^2 = 1,\, 5,\, 73,\,\ldots$$ reflect a VHS with monodromies of the same types as in $\S$\ref{SIVA} except at $t=\infty$ (where we get maximal unipotent monodromy). 
\end{minipage}
\begin{minipage}{0.3\textwidth}	
\[\begin{tikzpicture}[scale=0.8]
\node[black] at (1.2,-0.5) {$\sigma=\infty$};
\draw[gray,thick] (0,2.5) -- (0,0) -- (2.5,0);
\filldraw[black] (0,0) circle (3pt);
\filldraw[black] (1,1) circle (3pt);
\filldraw[black] (2,2) circle (3pt);
\draw[->,thick,blue] (1.8,1.8) -- (1.2,1.2);
\draw[->,thick,blue] (0.8,0.8) -- (0.2,0.2);
\node [blue] at (1.8,1.2) {$N_\infty$};
\draw[green,thick] (0,0) circle (6pt);
\end{tikzpicture}\]
\end{minipage}

Since $\phi$ is tempered, the symbol $\{\ux\}$ lifts to $\xi\in \mathrm{CH}^3(\cx\setminus \tilde{X}_0,3)$.  The birational map $\mathcal{I}\colon (\ux,t)\mapsto \left(\tfrac{x_3}{1-x_3},\tfrac{-(1-x_1)(1-x_2)}{1-x_1-x_2+x_1x_2-x_1x_2x_3},\tfrac{x_1}{1-x_1},\tfrac{1}{t}\right)$ from $\cx$ to itself, viewed as a correspondence, allows us to define $\mathcal{Z}:=\mathcal{I}^*\xi\in \mathrm{CH}^3(\cx\setminus\tilde{X}_{\infty},3)$.  The resulting THNF
\begin{multline*}
V(t)= \langle \tilde{\nu}_{\mathcal{Z}}(t),\omega_t\rangle \overset{\mathcal{I}}{=} \langle \tilde{\nu}_{\phi}(t^{-1}),t^{-1}\omega_{t^{-1}}\rangle=\int_{X_{t^{-1}}} R_3(\ux)\wedge d\left[ \tfrac{1}{(2\pi\ay)^3}\tfrac{\mathrm{dlog}(\ux)}{t-\phi(\ux)}\right] =\int d\left[\tfrac{R_3(\ux)}{(2\pi\ay)^3}\right] \wedge \tfrac{\mathrm{dlog}(\ux)}{\phi(\ux)-t} \\ = \int_{\RR_-^3}\tfrac{\mathrm{dlog}(\ux)}{t-\phi(\ux)} = -\sum_{k\geq 0}t^k \int_{\RR_-^3} \tfrac{\mathrm{dlog}(\ux)}{(\phi(\ux))^{k+1}}=\sum_{k\geq 0}t^k \left(\int_{[0,1]^3}\tfrac{\prod^3_{i=1} X_i^k (1-X_i)^k dX_i}{(1-X_3(1-X_1 X_2))^{k+1}} \right) =: \sum_{k\geq 0} v_k t^k
\end{multline*}
has $v_0=2\zeta(3)$ and $v_1=-12+10\zeta(3)$, which (again by Theorem \ref{t51}) is enough to conclude that $LV=-12t$.  But then normalization gives 
\begin{equation}\label{e54b}
\hat{V}(t)= \tfrac{1}{6}\zeta(3)+(-1+\tfrac{5}{6}\zeta(3))t+\cdots,	
\end{equation}
and in particular $\hat{V}(0)=\alpha$.  

\subsection{HNFs for $V_{14},V_{16},V_{18}$} \label{SIVC}

Again, the LG models are families of Picard-rank 19 $K3$s.  The irrational case $V_{14}=G(2,6)\cap\mathsf{P}_1\cap \cdots \cap \mathsf{P}_5$ is similar to $V_{10}$, with Laurent polynomial $$\phi(\ux)=\frac{(1-x_1-x_2-x_3)\{(1{-}x_2{-}x_3)(1{-}x_3)^2 -x_2(1{-}x_1{-}x_2{-}x_3)\}}{-x_1 x_2 x_3},$$ discriminant locus $\Sigma =\{0,\tfrac{1}{27},-1,\infty\}$, and PF operator $$L=\delta_t^3 - t(1+2\delta_t)(13\delta_t^2+13\delta_t+4)-3t^2(\delta_t+1)(3\delta_t+4)(3\delta_t+2).$$  The monodromy types are the same as for $V_{10}$, except for $T_{\infty}$, which acts on $(\psi_{\infty}\ch^2_v)^{2,0}$ resp. $(\psi_{\infty}\ch^2_v)^{0,2}$ by $e^{-\frac{2\pi\ay}{3}}$ resp. $e^{\frac{2\pi\ay}{3}}$.

The toric boundary divisor $x_1=0$ intersects $X_t$ in $\mathcal{C}_1=\{x_2=1-x_3\}$ and $\mathcal{C}_2=\{x_2=(1-x_3)^2\}$, and a cycle $\mathcal{Z}\in \mathrm{CH}^2(\cx{\setminus}\tilde{X}_{\infty},1)$ is given by $(\mathcal{C}_1,\tfrac{x_3}{1-x_3})+(\mathcal{C}_2,\tfrac{1-x_3}{x_3})$. Arguing as before, this yields $$V(t)=\sum_{k\geq 0}v_k t^k = \sum_{k\geq 0}t^k \int_{\mu} [\phi^k]_{x_1^0} \tfrac{dx_2}{x_2}\wedge \tfrac{dx_3}{x_3}\,,$$ where $\mu:=\{(x_1,x_2)\in\RR^2\mid 0\leq x_3\leq 1,\;(1-x_3)^2\leq x_2\leq 1-x_3\}$.  We compute $v_0=\zeta(2)$ and 
$$v_1=\int_0^1 \int_{(1-x_3)^2}^{1-x_3} \left\{\begin{matrix}2x_2 x_3^{-1}+(-x_3+4-3x_3^{-1})\\+x_2^{-1}x_3^{-1}(x_3-1)^3\end{matrix}\right\}\frac{dx_2}{x_2}\wedge \frac{dx_3}{x_3}=-7+4\zeta(2)\,,$$
hence that $LV=-7t$.  Renormalizing this gives $$\hat{V}(t)=\tfrac{1}{7}\zeta(2)+(-1+\tfrac{4}{7}\zeta(2))t+\cdots\;,$$ and $\hat{V}(0)=\tfrac{1}{7}\zeta(2)$ indeed matches the $\alpha$ from \cite{Gv}.

Turning to $V_{16}=\mathit{LG}(3,6)\cap\mathsf{P}_1 \cap\mathsf{P}_2\cap\mathsf{P}_3$ and $V_{18}=(G_2/P_2)\cap \mathsf{P}_1\cap\mathsf{P}_2$, we use
\begin{align*}
\phi(\ux) = \tfrac{(1-x_1-x_2-x_3)(1-x_1)(1-x_2)(1-x_3)}{-x_1x_2x_3}\;\;\;\text{resp.}\;\;\; \tfrac{(x_1+x_2+x_3)(x_1+x_2+x_3-x_1x_2-x_2x_3-x_1x_3+x_1x_2x_3)}{-x_1x_2x_3}
\end{align*}
\noindent from \cite{dS2} for our LG models, with $\Sigma=\{0,12\pm 8\sqrt{2},\infty\}$ resp. $\{0,9\pm 6\sqrt{3},\infty\}$ and
\begin{multline*}
L = \delta_t^3 - 4t(1+2\delta_t)(3\delta_t^2+3\delta_t+1)+16t^2(\delta_t+1)^3\\
\text{resp.}\; \delta_t^3 - 3t(1+2\delta_t)(3\delta_t^2+3\delta_t+1)-27t^2(\delta_t+1)^3.
\end{multline*}
\noindent The monodromy/LMHS types are the same as for $V_{12}$; we write $N=6$, $8$, or $9$ for $V_{2N}$, and put $I(t):=\tfrac{1}{M_N t}$ with $M_N=1,\tfrac{1}{16},\tfrac{-1}{27}$ respectively.  In each case there is an isomorphism $\ch_v^2\cong I^*\ch^2_v$ of $\QQ$-VHS.\footnote{This is easiest to see from the differential equation, but also follows from the fact that (for all five cases) the LG model of $V_{2N}$ realizes the canonical weight-$2$ rank-$3$ VHS over $X_0(N)^{+N}$, which for $N$ composite has an additional Fricke involution.}  For $N=8,9$ this is not an integral isomorphism so is induced by correspondences $\mathcal{I},\mathcal{I}^{-1}\in Z^2(\cx\times I^*\cx)_{\QQ}$ (with $\ci^*\circ (\ci^{-1})^*=\text{id}_{\ch^2_v}$) rather than a birational map; nevertheless, we may still define $\cz:=\ci^*\xi\in \mathrm{CH}^3(\cx{\setminus}\tx_{\infty},3)$.  Here we normalize $\ci$ to pull back an integral generator $\zeta_s$ of $(\HH_v^2)^{T_{\infty}}$ back to $\gamma_t\in (\HH^2_v)^{T_0}$, where $s=I(t)$.

Since the integrals $\int_{\RR_-^3}\tfrac{\text{dlog}(\ux)}{(\phi(\ux))^{k+1}}$ are quite difficult for $N=8,9$, we use a different strategy than in \S\ref{SIVB}.  As a section of $\ch^{2,0}_{v,e}\cong\co_{\PP^1}(1)$, $\omega_t$ has divisor $[\infty]$, and so $(\ci^{-1})^*\omega_{I(t)}=C_N t\omega_t$ for some $C_N\in\CC^*$.  Write $\zeta_s^{\vee}\in (\HH_v^2)_{T_{\infty}}$ for the element dual to $\zeta_s$, so that $\lim_{s\to\infty} s\omega_s = \tfrac{-1}{(2\pi\ay)^2}\mathrm{Res}_{\tx_{\infty}}(\tfrac{\text{dlog}(\ux)}{\phi(\ux)})=\mathfrak{r}_N\zeta_{\infty}^{\vee}$ in $H_2(\tx_{\infty})$ where $\mathfrak{r}_N:=\tfrac{-1}{(2\pi\ay)^3}\mathrm{Res}^3_p(\tfrac{\text{dlog}(\ux)}{\phi(\ux)})=1,\,\tfrac{1}{2}$, resp. $\tfrac{1}{\sqrt{-3}}$ (for some triple-normal-crossing point $p\in\tx_{\infty}$).  This yields
$$C_N = \lim_{t\to 0}\tfrac{1}{t}\langle \gamma_t,(\ci^{-1})^*\omega_{I(t)}\rangle =\lim_{s\to\infty}M_N s\langle(\ci^{-1})^* \gamma_{I(s)},\omega_s\rangle = \lim_{s\to\infty} M_N\langle \zeta_s,s\omega_s\rangle = M_N\mathfrak{r}_N.$$
Write $\Lambda$ for $L$ with $t$ replaced by $s$, we have $L=\tfrac{-1}{M_N s}\Lambda \tfrac{1}{s}$.  Applying this to
\begin{equation}\label{e55a}
V(t)= \langle \tilde{\nu}_{\cz}(t),\omega_t\rangle =\tfrac{1}{C_N t}\langle \ci^* \tilde{\nu}_{\phi}(s),(\ci^{-1})^*\omega_s\rangle =\tfrac{M_N s}{C_N}\langle \tilde{\nu}_{\phi}(s),\omega_s\rangle
\end{equation}
yields $LV=\tfrac{-1}{C_N s}\Lambda \langle \tilde{\nu}_{\phi}(s),\omega_s\rangle=\tfrac{-D_N}{C_N s}=-\tfrac{D_N}{\mathfrak{r}_N} t$, where $D_N=12,\,16,$ resp. $9$ is the constant from Remark \ref {r43b}.  Moreover, thinking of $\zeta_{\infty}^{\vee}$ as a ``membrane stretched once around $X_{\infty}$'', taking $\lim_{s\to \infty}$ of \eqref{e55a} gives
\begin{equation}\label{e55b}
V(0)=-\tfrac{1}{\mathfrak{r}_N}\int_{X_{\infty}} R_3(\ux)\wedge \tfrac{1}{(2\pi\ay)^2}\mathrm{Res}_{X_{\infty}}\left( \tfrac{\text{dlog}(\ux)}{\phi(\ux)}\right) = \int_{\zeta_{\infty}^{\vee}} R_3(\ux)|_{X_{\infty}}
\end{equation}
for $\zeta_{\infty}^{\vee}$ in suitably general position;\footnote{Compare \cite[Thm. 4.2(b) + Cor. 4.3]{Ke17}, which this generalizes.} and $\hat{V}(0)=\tfrac{\mathfrak{r}_N}{D_N}V(0)$.

We now use \eqref{e55b} to verify that $\hat{V}(0)$ recovers the Ap\'ery constants in \cite{Gv}.  For $N=6$, the computation in \cite[\S5.3]{Ke17} (with $\zeta_{\infty}^{\vee}=-\psi$) gives $\int_{\zeta_{\infty}^{\vee}}R_3(\ux)=2\zeta(3)$, recovering $\hat{V}(0)=\tfrac{\zeta(3)}{6}$.  For $N=8$, putting $\zeta_{\infty}^{\vee}$ in general position is tricky so we use the first expression in \eqref{e55b}. As $R_3(\ux)=\log(x_1)\tfrac{dx_2}{x_2}\wedge\tfrac{dx_3}{x_3}+2\pi\ay\log(x_2)\tfrac{dx_3}{x_3}\delta_{T_{x_1}}+(2\pi\ay)^2\log(x_3)\delta_{T_{x_1}\cap T_{x_2}}$ is nontrivial only on the component $\{x_1=1-x_2-x_3\}\subset X_{\infty}$, with only its third term surviving against the $(2,0)$ residue form, this yields 
\begin{multline*}
V(0) = -2\int_{T_{1-x_2-x_3}\cap T_{x_2}}\log(x_3)\tfrac{dx_2\wedge dx_3}{(1-x_2)(1-x_3)(x_2+x_3)} =-2\int_1^{\infty}\tfrac{\log(x_3)}{1-x_3}\left(\int_{1-x_3}^0 \tfrac{dx_2}{(1-x_2)(x_2+x_3)}\right)dx_3\\ \underset{u=x_3^{-1}}{=} 4 \int_0^1 \frac{\log^2(u)}{1-u^2}du =4(\mathrm{Li}_3(1)-\mathrm{Li}_3(-1))=7\zeta(3),
\end{multline*}
\noindent hence $\hat{V}(0)=\tfrac{7}{32}\zeta(3)$.

Finally, for $N=9$ we first replace $\{\ux\}$ (hence $\xi$, and $\cz$) by the equivalent symbol $\{\uz\}$, where $z_1=\tfrac{-x_3}{x_1+x_2}$, $z_2=-\tfrac{x_1}{x_2}$, and $z_3=\tfrac{x_1x_2}{x_1+x_2}$. In these coordinates, $$\phi(\ux(\uz))=z_1^{-1}z_3^{-1}(1-z_1)\{(1-z_3)-z_1(1-(1-z_2)z_3)(1-(1-z_2^{-1})z_3)\}$$ and so $X_{\infty}=X_{\infty}'\cup X_{\infty}''=\{z_1=1\}\cup \{z_1=\varphi(z_2,z_3):=\tfrac{1-z_3}{(1-(1-z_2)z_3)(1-(1-z_2^{-1})z_3)}\}$, with $\cC_{\infty}:=X_{\infty}'\cap X_{\infty}''$ described by $z_3=\mathfrak{z}(z_2):=\tfrac{1-z_2-z_2^{-1}}{(1-z_2)(1-z_2^{-1})}$. Clearly $R_3(\ux)|_{X_{\infty}'}=0$. For $\zeta_{\infty}^{\vee}\cap X_{\infty}''$, which must bound on $\cC_{\infty}$, we may take the 2-chain parametrized by $(z_2,z_3)=\{(e^{\ay\theta},\rho\mathfrak{z}(e^{\ay\theta}))\mid \theta\in[-\tfrac{\pi}{3},\tfrac{\pi}{3}],\,\rho\in[0,1]\}$. This yields
\begin{multline*}
V(0)= \int_{\zeta_{\infty}^{\vee}\cap X_{\infty}''} \log(\varphi(z_2,z_3))\frac{dz_2}{z_2}\wedge\frac{dz_3}{z_3} = \int_{\partial(\zeta_{\infty}^{\vee}\cap X_{\infty}'')} \left( \begin{matrix}\mathrm{Li}_2(\mathfrak{z}(z_2))-\mathrm{Li}_2((1-z_2)\mathfrak{z}(z_2))\\-\mathrm{Li}_2((1-z_2^{-1})\mathfrak{z}(z_2))\end{matrix}\right)\frac{dz_2}{z_2}\\
=\int_{e^{-\frac{\pi\ay}{3}}}^{e^{\frac{\pi\ay}{3}}}\left( 4\log(1-u)+\log(u)\right) \log(u)\frac{du}{u} = \left[ 4\mathrm{Li}_3(u)-4\mathrm{Li}_2(u)\log(u)+\tfrac{1}{3}\log^3(u)\right]_{e^{-\frac{\pi\ay}{3}}}^{e^{\frac{\pi\ay}{3}}}=\tfrac{4\pi^3\ay}{27},
\end{multline*}
\noindent whereupon $\hat{V}(0)=\tfrac{V(0)}{9\sqrt{-3}}=\tfrac{4\pi^3}{3^5\sqrt{3}}=\tfrac{1}{3}L(\chi_3,3)$.

\section{Ap\'ery and normal functions}\label{SV}

In this brief final section, we introduce a framework for studying the normal functions arising in connection with the Arithmetic Mirror Symmetry Problem (including the examples in \S\ref{SIV}), and propose some terminology.

\begin{defn}\label{d6a}
The \emph{Ap\'ery motive} is $\mathsf{A}_{\phi}:=H^n(\cx\setminus \tilde{X}_{\infty},\tilde{X}_0)/\mathrm{Ph}_{\Sigma^*}$ from \eqref{e44c}, or (if one prefers) its underlying mixed motive.
\end{defn}

\noindent We dig into its structure a bit:  there are exact sequences of MHS
\begin{equation}\label{e61}
\xymatrix@R=0.7em{&&& 0 \\ &&& (\psi_{\infty}\ch_v^{n-1})_{T_{\infty}}(-1) \ar [u] \\ 0 \ar [r] & (\psi_0\ch_v^{n-1})^{T_0} \ar [r] & \mathsf{A}_{\phi} \ar [r] & \IH^1(\mathbb{A}^1,\ch_v^{n-1}) \ar [r] \ar [u] & 0 \\ &&& \IH^1(\PP^1,\ch^{n-1}_v) \ar [u] \\ &&& 0 \ar [u] }
\end{equation}
where $(\cdot)_{T_{\infty}}=\mathrm{coker}(T_{\infty}-I)$, $(\cdot)^{T_0}=\ker(T_0 - I)$, and $\mathbb{A}^1$ means $\PP^1\setminus \{t=\infty\}$.  The parabolic cohomology $\IH^1(\PP^1,\ch^{n-1}_v)$ is pure of weight $n$ and rank 
\begin{equation}\label{e62}
{ih}^1(\PP^1,\ch^{n-1}_v)=\sum_{\sigma\in\Sigma}\mathrm{rk}(T_{\sigma}-I) - 2r.
\end{equation}
\begin{defn}\label{d6b}
$\ch^{n-1}_v$ (or $\phi$) is \emph{extremal} if \eqref{e62} is zero.
\end{defn}
\noindent Recall that if $\phi$ is tempered, the coordinate symbol $\{\ux\}$ lifts to $\xi\in \mathrm{CH}^n(\cx\setminus \tilde{X}_0,n)$.  The cycle class of $\mathrm{Res}(\xi)\in\mathrm{CH}^{n-1}(\tilde{X}_0,n-1)$ yields an embedding $\QQ(-n)\hookrightarrow H^{n+1}_{\tilde{X}_0}(\cx)\cong H_{n-1}(\tilde{X}_0)(-n)$, or dually\footnote{The first map in the portion $H_{n+1}(\tilde{X}_0)(-n)\to H^{n-1}(\tilde{X}_0)\to (\psi_0\ch^{n-1}_v)^{T_0}\to 0$ of the Clemens-Schmid sequence has pure weight $n-1$, and so the second map has a splitting $(\psi_0 \ch^{n-1}_v)^{T_0}\hookrightarrow H^{n-1}(\tilde{X}_0)$ which is an isomorphism in weights $<n-1$.  Dualizing the embedding yields $H^{n-1}(\tilde{X}_0)\twoheadrightarrow \QQ(0)$, and \eqref{e63} is the composition.} a splitting
\begin{equation}\label{e63}
\varepsilon\colon (\psi_0\ch_v^{n-1})^{T_0}\twoheadrightarrow \QQ(0).
\end{equation}

Suppose then that $\phi$ is tempered and extremal, and that (for some $p\in\mathbb{N}$) there exists an embedding
\begin{equation}\label{e64}
\mu\colon \QQ(-p)\hookrightarrow (\psi_{\infty}\ch^{n-1}_v)_{T_{\infty}}(-1).
\end{equation}
\noindent Then from \eqref{e61} we obtain the diagram
\begin{equation}\label{e65}
\xymatrix@R=1em{0\ar[r] & \QQ(0) \ar @{=} [d] \ar [r] &\mu^*\varepsilon_* \mathsf{A}_{\phi} \ar [r] \ar @{^(->} [d] & \QQ(-p) \ar @{^(->} [d]^{\mu} \ar [r] & 0 \\ 0 \ar [r] & \QQ(0) \ar [r] & \varepsilon_*\mathsf{A}_{\phi} \ar [r] & (\psi_{\infty}\ch^{n-1}_v)_{T_{\infty}}(-1) \ar [r] & 0 \\ 0 \ar [r] & (\psi_0 \ch_v^{n-1})^{T_0} \ar @{->>} [u]^{\varepsilon} \ar [r] & \mathsf{A}_{\phi} \ar [r] \ar @{->>} [u] & \IH^1(\mathbb{A}^1,\ch^{n-1}_v) \ar @{=} [u] \ar [r] & 0} 
\end{equation}
\noindent with exact rows.  Under $\mathrm{Ext}^1_{\text{MHS}}(\QQ(-p),\QQ(0))\cong \CC/\QQ(p)$, define
\begin{equation}\label{e66}
\alpha_{\phi}(\mu)\in \CC/\QQ(p)
\end{equation}
\noindent to be the image of the extension class of the top row of \eqref{e65}.

\begin{example}\label{e6a}
The Laurent polynomials considered in \S\S\ref{SIVA}-\ref{SIVC} are tempered and extremal, with $(\psi_{\infty}\ch^2_v)_{T_{\infty}}(-1)\cong \QQ(-2)$ resp. $\QQ(-3)$ for $V_{10},V_{14}$ resp $V_{12},V_{16},V_{18}$.  (Indeed, $\mu$ and $\varepsilon$ are both isomorphisms.)  In view of \eqref{e610} below, in each case $\alpha_{\phi}(\mu)$ is just $V(0)$ viewed modulo $\QQ(p)$.  But for $V_{10}$, $V_{14}$, and $V_{18}$, $V(0)$ is \emph{in} $\QQ(p)$ and so $\alpha_{\phi}(\mu)$ is trivial!
\end{example}

From the example we see the importance of presenting \eqref{e66} \emph{as a limit of a HNF}, since by canonically normalizing the latter (Remark \ref{r51a}) we may then refine \eqref{e66} to a well-defined complex number.  To do this, note that the same proof as for Proposition \ref{p44a} expresses the VMHS $\mathcal{A}_{\phi,t}:=H^n(\cx\setminus \tilde{X}_{\infty},\tilde{X}_t)/\mathrm{Ph}_{\Sigma\setminus\{\infty\}}$ as an extension
\begin{equation}\label{e67}
0\to \ch_v^{n-1}\to \mathcal{A}_{\phi}\to \IH^1(\mathbb{A}^1,\ch^{n-1}_v)\to 0\,.
\end{equation}
So we arrive at this article's eponymous
\begin{defn}\label{d6c}
The pullback 
\begin{equation}\label{e68}
0 \to \ch^{n-1}_v \to \mu^* \mathcal{A}_{\phi} \to \QQ(-p)\to 0
\end{equation}
of \eqref{e67} under \eqref{e64} is called an \emph{Ap\'ery extension}.
\end{defn}
\noindent We may view \eqref{e68} as a higher normal function
\begin{equation}\label{e69}
\nu_{\mu}\in \mathrm{ANF}(\ch^{n-1}_v(p))
\end{equation}
\noindent which is singular at $t=\infty$ and only there,\footnote{see the proof of \cite[Thm. 10.8]{Ke20}.} and we define and normalize $V_{\mu}(t):=\langle \nu_{\mu,t},\omega_t\rangle$ as in \S\ref{SIV1}.  Since $\mu^*\mathsf{A}_{\phi}\cong (\psi_0\mu^*\mathcal{A}_{\phi})^{T_0}$ and $\varepsilon$ is induced by pairing with $\lim_{t\to 0}\omega_t$, we conclude that 
\begin{equation}\label{e610}
V_{\mu}(0)\mapsto \alpha_{\phi}(\mu)\;\;\;\;\;\text{under}\;\;\;\;\;\CC\twoheadrightarrow\CC/\QQ(p).
\end{equation}
\noindent At least when $\phi$ is tempered and extremal, and $\IH^1(\mathbb{A}^1,\ch^{n-1}_v)$ is split Hodge-Tate, it is these $V_{\mu}(0)$ which are expected to produce (up to $\bar{\QQ}^*$) the interesting Ap\'ery constants.  (In the absence of these conditions, of course, the situation will be more complicated.)

Conversely, any $\nu\in \mathrm{ANF}(\ch_v^{n-1}(p))$ nonsingular off $\infty$ arises as in \eqref{e68}.  In our fine examples, $\IH^1(\mathbb{A}^1,\ch^2_v)\cong (\psi_{\infty}\ch^2_v)_{T_{\infty}}(-1)$ has rank one.  So this finishes off the uniqueness part of Problem \ref{c52} in each case, completing the proof of Theorem \ref{tm}.

Now according to the BHC (since $\cx{\setminus}\tilde{X}_{\infty}$ is defined over $\bar{\QQ}$), there exists a higher cycle $\mathcal{Z}_{\mu}\in\mathrm{CH}^p(\cx{\setminus}\tilde{X}_{\infty},2p-n)_{\QQ}$ with $\nu_{\mu}=\nu_{\mathcal{Z}_{\mu}}$.  This provides a mechanism for explaining the arithmetic content of the Ap\'ery constants.  Let $K\subset \bar{\QQ}$ be the field of definition of $\mathcal{Z}_{\mu}$.  By \cite[Cor. 5.3ff]{7K}, $V_{\mu}(0)$ may be interpreted as the image of $\mathcal{Z}_0:=\imath^*_{\tilde{X}_0}\mathcal{Z}_{\mu}$ under 
\begin{equation}\label{e611}
H^n_{\mathcal{M}}(\tilde{X}_{0,K},\QQ(p))\overset{\mathrm{AJ}}{\to} J(H^{n-1}(\tilde{X}_0)(p))\overset{\langle\cdot,\omega_0\rangle}{\twoheadrightarrow}\CC/\QQ(p)\,,
\end{equation}
\noindent where the second map comes from temperedness of $\phi$.  Since $\tilde{X}_0=\cup_i Y_i$ is a NCD, we have a spectral sequence $E_1^{a,b}=Z^p(\tilde{X}^{[a]}_0,2p-n-b)\,\underset{a+b=*}{\implies}H_{\mathcal{M}}^{n+*}(\tilde{X}_0,\QQ(p))$ where $\tilde{X}_0^{[a]}:=\amalg_{|I|=a+1}(\cap_{i\in I}Y_i)$.  The induced filtration $\mathscr{W}_{\bullet}$ [op. cit., \S3] has bottom piece 
\begin{align*}
\mathscr{W}_{-n+1}H_{\mathcal{M}}^n(\tilde{X}_{0,K},\QQ(p)) &\cong \mathrm{coker}\{\mathrm{CH}^p(\tilde{X}_{0,K}^{[-n+2]},2p-1)\to \mathrm{CH}^p(\tilde{X}_{0,K}^{[-n+1]},2p-1)\} \\
&\cong \mathrm{CH}^p(\mathrm{Spec}(K),2p-1)\,,
\end{align*}
\noindent and \eqref{e611} restricts to the Borel regulator on this piece.

\begin{example}\label{e6b} 
In \S\S\ref{SIVA}-\ref{SIVC}, $\mathcal{Z}_0$ belongs to $\mathscr{W}_{-2}H^3_{\mathcal{M}}(\tilde{X}_{0,K},\QQ(p))$, with $K=\QQ(\sqrt{-3})$ for $V_{18}$ and $K=\QQ$ for the other $V_{2N}$'s.  Since each $\alpha_{\mathrm{X}^{\circ}}$ is also real by construction, and $\mathfrak{k}$ belongs to $K$, Borel's theorem (together with part (c) of Problem \ref{c52}) \emph{forces} $\alpha_{\mathrm{X}^{\circ}}$ to be in $\QQ(2)$ ($V_{10},V_{14}$), $\zeta(3)\QQ$ ($V_{12},V_{16}$), and $\sqrt{-3}\QQ(3)$ ($V_{18}$) respectively, before any computation is done.
\end{example}

\begin{rem}\label{r6a}
We finally owe the reader an explanation regarding the flip in perspective from $\psi_{\infty}\mathcal{A}_{\phi}^{\dagger}$ (and the limit of the coordinate-symbol normal function at $\infty$) to $\psi_0\mathcal{A}_{\phi}$ (and the limit of Ap\'ery normal functions at $0$), specifically the ``computational nonviability'' claimed for the former.  First of all, if we choose the lift $\tilde{\nu}_{\phi}$ to be single-valued around $\infty$, it is a section of the dual-canonical extension $(\ch^{n-1}_v)^e$; since $\omega$ is a section of $\ch^{n-1}_{v,e}$ with a simple zero at $\infty$, they do indeed pair to a holomorphic function $V_{\phi}$ on a disk $D_{\infty}$, but one with $V_{\phi}(\infty)=0$.  Replacing $\omega$ by $\hat{\omega}:=t\omega$ gives $\hat{V}_{\phi}:=tV_{\phi}$, from which we can in principle read off the limiting extension class if we know the limits of the invariant periods of $\hat{\omega}$ at $\infty$ (which is already nontrivial); and this was essentially the method used for $V_{16}$ and $V_{18}$.  

But this approach becomes problematic when $T_{\infty}$ is non-unipotent, intuitively because $\hat{\omega}$ then has periods which blow up at $\infty$, and we lack a suitable representative for $\hat{\omega}(\infty)$ on $\tilde{X}_{\infty}$.  More precisely, if we let $\rho\colon D_{\infty}\to D_{\infty}$ be the base-change (ramified at $\infty$) which kills $T_{\infty}^{\text{ss}}$, the pullback $\rho^*\omega$ is \emph{not} a section of $(\rho^*\ch_v^{n-1})_e$ over $D_{\infty}$, and we cannot use \cite[Cor. 5.3]{7K} to compute $\hat{V}_{\phi}(\infty)$.  So while, for (say) $V_{10}$ and $V_{14}$, one can show (abstractly, from its inhomogeneous equation) that $\hat{V}_{\phi}(\infty)$ is a nonzero complex number, it does not seem nearly as accessible as the $V_{\mu}(0)$ values computed in \S\S\ref{SIVA} and \S\S\ref{SIVC}.
\end{rem}

\begin{rem}\label{r6b}
In addition to its implications for the arithmetic of $\alpha_{\mathrm{X}^{\circ}}$, Problem \ref{c52} appears to encode interesting algebro-geometric predictions about Fanos.  To just give the idea in the simplest possible case, suppose $\mathsf{F}^m$ is a Fano $m$-fold, with $H^*_{\text{prim}}(\mathsf{F}^m)$ of rank two, concentrated in weights $0$ and $2w$, with degree $2$ QDE.  Let $\mathsf{P_i}$ denote general hyperplanes in some $\PP^M$ in which $\mathsf{F}^m$ is minimally embedded, and accept the idea that --- as long as $\mathsf{F}^{m-\ell}:=\mathsf{F}^m\cap \mathsf{P}_1\cap\cdots\cap\mathsf{P}_{\ell}$ remains Fano --- \ref{c52}(a)-(c) continue to hold and the Ap\'ery constant $\alpha$ remains unchanged (cf. Remark \ref{r33a}).  At first, $\alpha=\alpha_{\mathsf{F}^m}$ is computed by the LMHS of the LG model; but after hyperplane sections kill off the second Lefschetz string in $H^*$, $\alpha=\alpha_{\mathsf{F}^{m-\ell}}$ is computed by the limit of a nontorsion extension of $\ch^{m-\ell-1}$ by $\QQ(-w)$ (even if $\alpha$ is ``torsion'').  As soon as $F^{w-1}\ch^{m-\ell-1}=\{0\}$, however, Griffiths transversality forces such normal functions to be flat and thus torsion.  So for Problem \ref{c52} to be consistent, $\mathsf{F}^{m-\ell}$ cannot be Fano for $m-\ell<w$; that is, the \emph{index} $\mathtt{i}(\mathsf{F}^m)$ is $\leq m-w+1$.  (Recall that $\mathtt{i}$ is defined by $-K_F=\mathtt{i}\mathsf{h}$.)  A quick perusal of examples in this paper suggests that this is sharp:  $G(2,5)$, $\mathit{OG}(5,10)$, and $\mathit{LG}(3,6)$ (but not $G(2,6)$ or $G_2/P_2$) each have rank two $H_{\text{prim}}^*$ and $d=2$; while their respective $(m,w,\mathtt{i})$ are $(6,2,5)$, $(10,3,8)$, resp. $(6,3,4)$.
\end{rem}

\end{document}